%% file: SCAH-arxiv2.tex
\documentclass[11pt,letterpaper]{article}
\usepackage{amsfonts, amsmath, amssymb, amscd, amsthm, graphicx,enumerate}
\usepackage{epsfig, color}
\usepackage{hyperref}

\makeatletter
\def\blfootnote{\xdef\@thefnmark{}\@footnotetext}
\makeatother

\newtheorem{thm}{Theorem}[section]
\newtheorem{cor}[thm]{Corollary}
\newtheorem{lem}[thm]{Lemma}
\newtheorem{prop}[thm]{Proposition}

\theoremstyle{definition}
\newtheorem{defn}[thm]{Definition}
\theoremstyle{remark}
\newtheorem{rem}[thm]{Remark}

\newfont{\eufm}{eufm10}

\newcommand{\G}{\Gamma (G, X\sqcup \mathcal H)}

\newcommand{\Hl}{\{ H_\lambda \} _{\lambda \in \Lambda } }
\newcommand{\e}{\varepsilon }
\renewcommand{\phi}{\varphi}

\newcommand{\lab}{{\bf Lab}}

\newcommand{\Gs}{\Gamma (G, S)}
\renewcommand{\l}{{\ell}}
\newcommand{\Ga }{\Gamma (G, \mathcal A)}
\newcommand{\N}{\mathbb N}
\newcommand{\Z}{\mathbb Z}

\renewcommand{\ll }{\langle\hspace{-.7mm}\langle }
\newcommand{\rr }{\rangle\hspace{-.7mm}\rangle }

\newcommand{\h}{\hookrightarrow _h }

\newcommand{\dl }{\widehat{\rm d}_{\lambda}}
\newcommand{\Lab }{{\bf Lab}}
\newcommand{\X }{\mathcal A\mathcal H }
\newcommand{\he }{hyperbolically embedded }
\newcommand{\Qp}{\{ Q_p\} _{p\in \Pi}}
\newcommand{\dol }{{\rm d}_{\Omega _\lambda}}

\def\twoheaddownarrow{\ensuremath{\rotatebox[origin=c]{-90}{$\twoheadrightarrow$}}}

\begin{document}

\title{Small cancellation in acylindrically hyperbolic groups}
\author{M. Hull}

\date{}
\maketitle

\begin{abstract}
We generalize a version of small cancellation theory to the class of acylindrically hyperbolic groups. This class contains many groups which admit some natural action on a hyperbolic space, including non-elementary hyperbolic and relatively hyperbolic groups, mapping class groups, and groups of outer automorphisms of free groups. Several applications of this small cancellation theory are given, including to Frattini subgroups and Kazhdan constants, the construction of various ``exotic" quotients, and to approximating acylindrically hyperbolic groups in the topology of marked group presentations.
\end{abstract}

\tableofcontents



\section{Introduction}
The idea of generalizing classical small cancellation theory to groups acting on hyperbolic spaces originated in Gromov's paper \cite{Gr2}. Gromov was motivated by the fact that hyperbolicity had been used \emph{implicity} in the ideas of small cancellation theory going back to the work of Dehn in the early 1900's; he claimed that many small cancellation arguments could be simultaneously simplified and generalized by \emph{explicitly} using hyperbolicity. In particular, Gromov showed how some of the ``exotic" groups constructed through complicated small cancellation arguments could be built as quotients of hyperbolic groups by inductively applying the following theorem:

\begin{thm}\cite{Gr2, Ols}\label{hypsc}
Let $G$ be a non-virtually-cyclic hyperbolic group, $F$ a finite subset of $G$, and $H$ a non-virtually-cyclic subgroup of $G$ which does not normalize any finite subgroups of $G$. Then there exists a group $\overline{G}$ and a surjective homomorphism $\gamma\colon G\to \overline{G}$ such that:
\begin{enumerate}
\item[(a)] $\overline{G}$ is a non-virtually-cyclic hyperbolic group.
\item[(b)] $\gamma|_F$ is injective.
\item[(c)] $\gamma|_H$ is surjective.
\item[(d)] Every element of $\overline{G}$ of finite order is the image of an element of $G$ of finite order.
\end{enumerate}
\end{thm}

In fact, Gromov's statement of this theorem was not correct, and after briefly sketching an argument for fundamental groups of manifolds, he said that the general case is ``straightforward and details are left to the reader" \cite{Gr2}. The correct statement and proof is due to Olshanskii, who actually proved a more general theorem by giving explicit combinatorial small cancellation conditions for hyperbolic groups and showing how to find words which satisfy those conditions \cite{Ols}. Applications and variations of Theorem \ref{hypsc} can be found in \cite{Gr2, MO, M, Ols, OlsS, Osi02, Oz}.

Building on the work of Olshanskii, Osin proved a version of Theorem \ref{hypsc} for relatively hyperbolic groups \cite{Osi10}. Using this, Osin gave the first construction of an infinite, finitely generated group with two conjugacy classes; this was the first known example of any finitely generated group with two conjugacy classes other than $\Z/2\Z$.  Other applications of Osin's version of Theorem \ref{hypsc} can be found in \cite{AMO, BelOsi, BelSz, HO1, M2, Osi10}.

The goal of this paper is to prove a version of Theorem \ref{hypsc} for a larger class of groups acting on hyperbolic metric spaces, specifically the class of \emph{acylindrically hyperbolic groups}.

\begin{defn}
Let $G$ be a group acting on a metric space $(X, d)$. We say that the action is \emph{acylindrical} if for all $\e>0$ there exist $R>0$ and $N>0$ such that for all $x$, $y\in X$ with $d(x, y)\geq R$, the set 
\[
\{g\in G\;|\; d(x, gx)\leq \e, d(y, gy)\leq \e\}
\]
contains at most $N$ elements.
\end{defn}

\begin{defn}
We say that a group $G$ is \emph{acylindrically hyperbolic} if $G$ admits a non-elementary, acylindrical action on some hyperbolic metric space. We denote the class of all acylindrically hyperbolic groups by $\X$.
\end{defn}

Recall that an action of $G$ on a hyperbolic space $X$ is called \emph{non-elementary} if $G$ has at least three limit points on the Gromov boundary $\partial X$. If $G$ acts acylindrically, this condition is equivalent to saying that $G$ is not virtually cyclic and some (equivalently, any) $G$-orbit is unbounded (see Theorem \ref{subah}).

The notion of an  acylindrical action was introduced for the special case of groups acting on trees by Sela \cite{S}, and in general by Bowditch studying the action of mapping class groups on the curve complex \cite{Bow2}. The term ``acylindrically hyperbolic" is due to Osin, who showed in \cite{Osip} that the class of acylindrically hyperbolic groups coincides with the class of groups which admit non-elementary actions on hyperbolic metric spaces which satisfy the WPD condition introduced by Bestvina-Fujiwara \cite{BF} and the class of groups which contain non-degenerate hyperbolically embedded subgroups introduced by Dahmani-Guirardel-Osin \cite{DGO}. It is not hard to see that any proper, cobounded action is acylindrical, and hence all non-virtually-cyclic hyperbolic groups belong to $\X$. Also, the action of a relatively hyperbolic group on the relative Cayley graph is acylindrical \cite{Osip}. Hence $\X$ is a generalization of the classes of non-virtually-cyclic hyperbolic and relatively hyperbolic groups. Other examples of acylindrically hyperbolic groups include:

\begin{enumerate}
\item The mapping class group of an orientable surface of genus $g$ with $p$ punctures for $3g+p>4$ \cite{Bow2}. (For $3g+p\leq 4$, this group is either non-virtually-cyclic hyperbolic or finite.)
\item $Out(F_n)$ for $n\geq 2$ \cite{DGO}.
\item Directly indecomposable non-cyclic right angled Artin groups, and more generally non-virtually-cyclic groups which act properly on proper $CAT(0)$ spaces and contain rank-1 elements \cite{Sis}.
\item The Cremona group of birational transformations of the complex projective plane \cite{DGO}.
\item The automorphism group of the polynomial algebra $k[x, y]$ for any field $k$ \cite{OM}.
\item All one relator groups with at least three generators \cite{OM}.
\item If $G$ is the graph product $G=\Gamma\{G_v\}_{v\in V}$ such that each $G_v$ has infinite index in $G$, then either $G$ is virtually cyclic, $G$ decomposes as the direct product of two infinite groups, or $G\in\X$ \cite{OM}.
\item If $G$ is the fundamental group of a compact 3-manifold, then either $G$ is virtually polycyclic, $G\in\X$, or $G/Z\in \X$ where $Z$ is infinite cyclic \cite{OM}.

\end{enumerate}

 For our version of Theorem \ref{hypsc}, we will need our chosen subgroup to be not only non-virtually-cyclic, but also non-elementary with respect to some acylindrical action on a hyperbolic metric space. By \cite{Osip}, this space can always be chosen to be a Cayley graph of $G$ with respect to some (possibly infinite) generating set $\mathcal A$ (see Theorem \ref{thmAH}). We denote this Cayley graph by $\Ga$, and we denote the ball of radius $N$ centered at the identity in $\Ga$ by $B_{\mathcal A}(N)$. As in Theorem \ref{hypsc}, we will require that our subgroup does not normalize any finite subgroups of $G$. Following the terminology of \cite{Osi10}, we will call such subgroups \emph{suitable}.

\begin{defn}
Given $G\in\X$, a generating set $\mathcal A$ of $G$ and a subgroup $S\leq G$, we will say that $S$ is \emph{suitable with respect to $\mathcal A$} if the following holds:
\begin{enumerate}
\item $\Ga$ is hyperbolic and the action of $G$ on $\Ga$ is acylindrical.
\item The induced action of $S$ on $\Ga$ is non-elementary.
\item $S$ does not normalize any finite subgroups of $G$.
\end{enumerate}
We will further say that a subgroup is \emph{suitable} if it is suitable with respect to some $\mathcal A$.
\end{defn}

\begin{thm}[Theorem \ref{scthm}]\label{intromain}
Suppose $G\in\X$ and $S\leq G$ is suitable with respect to $\mathcal A$. Then for any $\{t_1,...,t_m\}\subset G$ and $N\in\N$, there exists a group $\overline{G}$ and a surjective homomorphism $\gamma\colon G\to \overline{G}$ which satisfy
\begin{enumerate}
\item[(a)] $\overline{G}\in\X$.
\item[(b)] $\gamma|_{B_{\mathcal A}(N)}$ is injective.
\item[(c)] $\gamma(t_i)\in\gamma(S)$ for $i=1,...,m$.
\item[(d)] $\gamma(S)$ is a suitable subgroup of $\overline{G}$. 
\item[(e)] Every element of $\overline{G}$ of order $n$ is the image of an element of $G$ of order $n$.
\end{enumerate} 
\end{thm}

Typically, $\mathcal A$ will be an infinite subset of $G$, hence condition (b) is stronger in Theorem \ref{intromain} than in Theorem \ref{hypsc}. Indeed by choosing sufficiently large $N$, we can make $\gamma$ injective on any given finite set of elements. Also, if $G$ is finitely generated we can choose $t_1,...,t_m$ to be a generating set of $G$ and we get that $\gamma|_S$ is surjective; thus conditions (c) are equivalent in both theorems when $G$ is finitely generated. 

We will show that this theorem has a variety of applications, including the construction of various unusual quotient groups. In addition, this theorem allows us to easily generalize several results known for hyperbolic or relatively hyperbolic groups to the class of acylindrically hyperbolic groups.

We first record a useful corollary of our main theorem, which is a simplification of Corollary \ref{commonq}:
\begin{cor}\label{introcommonq}
Let $G_1,G_2\in\X$, with $G_1$ finitely generated, $G_2$ countable. Then there exists an infinite group $Q$ and surjective homomorphisms $\alpha_i\colon G_i\to Q$ for $i=1,2$. If in addition $G_2$ is finitely generated, then we can choose $Q\in\X$.
\end{cor}

Since Kazhdan's Property $(T)$ is preserved under taking quotients and the existence of infinite hyperbolic groups with Property $(T)$ is well-known, as an immediate consequense of Corollary \ref{introcommonq} we get:

\begin{cor}
Every countable $G\in\X$ has an infinite quotient with Property $(T)$.
\end{cor}

This generalizes a similar result of Gromov for non-virtually-cyclic hyperbolic groups \cite{Gr2}.

A version of Corollary \ref{introcommonq} for hyperbolic groups was used by Osin to study Kazhdan constants of hyperbolic groups \cite{Osi02}. Let $G$ be generated by a finite set $X$, and let $\pi\colon G\to \mathcal U(H)$ be a unitary representation of $G$ on a separable Hilbert space $H$. Then the \emph{Kazhdan constant of $G$ with respect to $X$ and $\pi$} is defined by
\[
\varkappa(G, X, \pi)=\inf_{\|v\|=1}\max_{x\in X}\|\pi(x)v-v\|.
\]
The \emph{Kazhdan constant of $G$ with respect to $X$} is the quantity 
\[
\varkappa(G, X)=\inf_\pi\varkappa(G, X, \pi)
\]
where this infimum is taken over all unitary representations that have no non-trivial invariant vectors. A finitely generated group has Property $(T)$ if and only if $\varkappa(G, X)>0$ for some (equivalently, any) finite generating set $X$. Lubotzky \cite{L} asked whether the quantity
\[
\varkappa(G)=\inf_X\varkappa(G, X)
\]
was non-zero for all finitely generated Property $(T)$ groups, where this infimum is taken over all finite generating sets of $G$. Clearly $\varkappa(G)>0$ if $G$ is finite, and examples of infinite, finitely generated groups $G$ with $\varkappa(G)>0$ were constructed in \cite{OsiSon}. However, a negative answer to Lubotzky's question was obtained by Gelander and \.Zuk  who showed that $\varkappa(G)=0$ whenever $G$ densely embeds in a connected, locally compact group \cite{GZ}. In addition, Osin showed that $\varkappa(G)=0$ whenever $G$ is an infinite hyperbolic group \cite{Osi02}. In fact, Osin actually proves that given any finitely generated group $G$, if for all non-virtually-cyclic hyperbolic groups $H$, $G$ and $H$ have a non-trivial common quotient, then $\varkappa(G)=0$. Thus combining Osin's proof and Corollary \ref{introcommonq}, we obtain the following.

\begin{thm}
Let $G\in\X$ be finitely generated. Then $\varkappa(G)=0$.
\end{thm}

Our next application is to the study of Frattini subgroups of groups in $\X$. The Frattini subgroup of a group $G$, denoted $Fratt(G)$, is defined as the intersection of all maximal proper subgroups of $G$, or as $G$ itself if no such subgroups exist. It is not hard to show that the Frattini subgroup of $G$ is exactly the set of \emph{non-generators} of $G$, that is the set of $g\in G$ such that for any set $X$ which generates $G$, $X\setminus\{g\}$ also generates $G$.

The study of the Frattini subgroup is related to the \emph{generation problem} and the \emph{rank problem}. Given  a group $G$ and a subset $Y\subseteq G$, the generation problem is to determine whether $Y$ generates $G$. The rank problem is to determine the smallest cardinality of a generating set of a given group $G$. Since $Fratt(G)$ consists of non-generators these problems can often be simplified by considering $G/Fratt(G)$. Hence these problems tend to be more approachable for classes of groups which have ``large" Frattini subgroups. We will show, however, that this is not the case for acylindrically hyperbolic groups.

\begin{thm}[Theorem \ref{thm:Fratt}]
Let $G\in\X$ be countable. Then $Fratt(G)$ is finite.
\end{thm}

This theorem generalizes several previously known results. For example, it was known that the free product of any non-trivial groups has trivial Frattini subgroup \cite{HN}, and that free products of free groups with cyclic amalgamation have finite Frattini subgroup \cite{Wa}. I. Kapovich proved that all subgroups of hyperbolic groups have finite Frattini subgroup \cite{K}, and Long proved that mapping class groups of closed, orientable surfaces of genus at least two have finite Frattini subgroup \cite{Lon}. All of these groups are either virtually cyclic or belong to $\X$.

Next we turn to the topology of marked group presentations. This topology provides a natural framework for studying groups which ``approximate" a given class of groups. For example, Sela's limit groups, which were used in the solution of the Tarski problem, can be defined as the groups which are approximated by free groups with respect to this topology (see \cite{CG}). In \cite{BE}, this topology is used to define a preorder on the space of finitely generated groups.

Let $\mathcal G_k$ denote the set of \emph{marked $k$-generated groups}, that is 
\[
\mathcal G_k=\{(G, X)\;|\; \text{$X\subseteq G$ is an ordered set of $k$ elements and } \langle X\rangle= G\}.
\]

This set is given a topology by saying that a sequence $(G_n, X_n)\rightarrow (G, X)$ in $\mathcal G_k$ if and only if there are functions $f_n\colon\Gamma(G_n, X_n)\to\Gamma(G, X)$ which are label-preserving isometries between increasingly large neighboorhoods of the identity. With this topology, $\mathcal G_k$ becomes a compact Hausdorff space.

 Given a class of groups $\mathcal X$, let $[\mathcal X]_k=\{(G, X)\in\mathcal G_k\;|\; G\in\mathcal X\}$. In case $\mathcal X$ consists of a single group $G$, we denote $[\mathcal X]_k$ by $[G]_k$. Also, let $[\mathcal X]=\bigcup_{k=1}^\infty[\mathcal X]_k$ and $\overline{[\mathcal X]}=\bigcup_{k=1}^\infty\overline{[\mathcal X]}_k$, where $\overline{[\mathcal X]}_k$ denotes the closure of $[\mathcal X]_k$ in $\mathcal G_k$.
 
 In the language of \cite{BE}, a group $H\in\overline{[G]}$ if and only if $G$ \emph{preforms} $H$, that is for some generating set $X$ of $H$ and some sequence of generating sets $X_1,...$ of $G$, 
\[
\lim_{n\rightarrow\infty}(G, X_n)=(H, X)
\]
where this limit is being taken in some fixed $\mathcal G_k$. In this situation, it is not hard to show that the universal theory of $G$ is contained in the universal theory of $H$ (see \cite{CG}).  Also,  note that any finite sub-structure of $\Gamma(H, X)$ can eventually be seen in $\Gamma(G, X_n)$; it follows that any quantifier-free first-order sentence which can be expressed using the language of groups and constants $\{x_1,..., x_k\}$ representing the elements of the ordered generating sets which holds in $(H, X)$ also eventually holds in $(G, X_n)$. It is for this reason that we think of $G$ as an ``approximation" of the group $H$. For example, if $W_1,...,W_m$ are words in $X$ and $W_1^\prime,...,W_m^\prime$ the corresponding words in $X_n$ such that $\{W_1,...,W_m\}$ is a finite (respectively, finite normal) subgroup of $H$, then for all sufficiently large $n$ $\{W_1^\prime,...,W_m^\prime\}$ is a finite (respectively, finite normal) subgroup of $G$.

From this perspective, the next theorem says that we can find a group $D$ which \emph{simultaneously} approximates countably many acylindrically hyperbolic groups. Let $\X_0$ denote the class of acylindrically hyperbolic groups which do not contain finite normal subgroups. Note that every $G\in\X$ has a quotient belonging to $\X_0$ (see Lemma \ref{k(g)}).

\begin{thm}[Theorem \ref{dense}]
Let $\mathcal C$ be a countable subset of $[\X_0]$. Then there exists a finitely generated group $D$ such that $\mathcal C\subseteq \overline{[D]}$.
\end{thm}

Finally, following constructions similar to those used by Osin in \cite{Osi10}, we are able to build some ``exotic" quotients of acylindrically hyperbolic groups. A group $G$ is called \emph{verbally complete} if for any $k\geq 1$, any $g\in G$, and any freely reduced word $W(x_1,..., x_k)$ there exist $g_1,...,g_k\in G$ such that $W(g_1,..., g_k)=g$ in the group $G$.  In particular, such groups are always \emph{divisible}, that is the equation $x^n=g$ has a solution in $G$ for all $n\in \Z\setminus\{0\}$ and all $g\in G$. The existence of non-trivial finitely generated verbally complete groups was shown by Mikhajlovskii and Olshanskii \cite{MO}, and Osin showed that every countable group could be embedded in a finitely generated verbally complete group \cite{Osi10}.

\begin{thm}[Theorem \ref{thm:vcquot}]
Let $G\in\X$ be countable. Then $G$ has a non-trivial finitely generated quotient $V$ such that $V$ is verbally complete.
\end{thm}

Higman, B. H. Neumann and H. Neumann showed that any countable group $G$ could be embedded in a countable group $B$ in which any two elements are conjugate if and only if they have the same order \cite{HNN}. Osin showed that the group $B$ could be chosen to be finitely generated \cite{Osi10}. We show that any countable  $G\in\X$ has such a quotient group. Here we let $\pi(G)\subseteq\N\cup\{\infty\}$ be the set of orders of elements of $G$.

\begin{thm}[Theorem \ref{conjquot}]\label{intro2cc}
Let $G\in\X$ be countable. Then $G$ has an infinite, finitely generated quotient $C$ such that any two elements of $C$ are conjugate if and only if they have the same order and $\pi(C)=\pi(G)$. In particular, if $G$ is torsion free, then $C$ has two conjugacy classes.
\end{thm}

Glassner and Weiss, motivated by the study of topological groups which contain a dense conjugacy class, asked about the existence of a topological analogue of the construction of a group with two conjugacy classes  \cite{GW}. Specifically, they asked about the existence of a non-discrete, locally compact topological group with two conjugacy classes. Combining the construction used in Theorem \ref{intro2cc} with the methods of \cite{OOK}, we can show the existence of a non-discrete, Hausdorff topological group with two conjugacy classes. Our methods do not give local compactness, but our group will be compactly and even finitely generated.

Given set $\mathcal S\subseteq \mathcal G_k$ and a group property $P$, we say a \emph{generic group in $\mathcal S$ satisfies $P$} if $\mathcal S$ contains a dense $G_\delta$ subset in which all groups satisfy $P$. A group $G$ is called \emph{topologizable} if $G$ admits a non-discrete, Hausdorff group topology; in \cite{OOK} it is proved that a generic group in $\overline{[\X]}_k$ is topologizable. Using the construction from the proof of Theorem \ref{intro2cc}, we can show (Corollary \ref{generic}) that a generic group in $\overline{[\X_{tf}]}_k$ has two conjugacy classes, where $\X_{tf}$ denotes the class of torsion free acylindrically hyperbolic groups. Since the Baire Category Theorem allows us to combine generic properties, we obtain the following. 

\begin{cor}\label{gentopcc}
For all $k\geq 2$, a generic group in $\overline{[\X_{tf}]}_k$ is topologizable and has two conjugacy classes. In particular, there exists a topologizable group with two conjugacy classes.
\end{cor}

The paper is organized as follows. In Section \ref{sect:2} we review some results about acylindrical and WPD actions. In Section \ref{sect:HES} we collect results about hyperbolically embedded subgroups, and in particular prove a sufficient condition for a collection of subgroups to be hyperbolically embedded with respect to a given generating set (Theorem \ref{crit}). In Section \ref{sectscq} we give properties of quotients of acylindrically hyperbolic groups which satisfy certain small cancellation conditions, and in Section \ref{sect:5} we characterize suitable subgroups and show that they contain sets of words satisfying the relevant small cancellation conditions. In Section \ref{HNN&amal} we show that suitable subgroups remain suitable after taking HNN-extensions or amalgamated products over cyclic subgroups. Finally, in Section \ref{sect:7} we prove Theorem \ref{scthm} as well as the various applications of this theorem mentioned in the introduction.

 \subsection*{Acknowledgments.} I would like to thank Denis Osin for providing guidance throughout the course of this project. I would also like to thank the anonymous referee for suggesting several improvements and corrections.

\section{Preliminaries}\label{sect:2}

\paragraph{\bf Notation} We begin by standardizing the notation that we will use. Given a group $G$ generated by a subset $S\subseteq G$, we denote by $\Gs $ the Cayley graph of $G$ with respect to $S$. That is, $\Gs$ is the graph with vertex set $G$ and an edge labeled by $s$ between each pair of vertices of the form $(g, gs)$, where $s\in S$. We will assume all generating sets are symmetric, that is $S=S\cup S^{-1}$. We let $|g|_S$ denote the \emph{word length} of an element $g$ with respect to $S$, that is $|g|_S$ is equal to the length of the shortest word in $S$ which is equal to $g$ in $G$. Similarly, $d_S$ will denote the \emph{word metric} on $G$ with respect to $S$,  that is $d_S(h,g)=|h^{-1}g|_S$. Clearly $d_S(h,g)$ is the length of the shortest path in $\Gs$ from $h$ to $g$.  We denote the ball of radius $n$ centered at the identity with respect to $d_S$ by $B_{S}(n)$; that is $B_{S}(n)=\{g\in G\;|\; |g|_S\leq n\}$. If $p$ is a (combinatorial) path in $\Gs$, $\lab (p)$ denotes its label, $\l(p)$ denotes its length, and $p_-$ and $p_+$ denote its starting and ending vertex.

In general, we will allow metrics and length functions to take infinite value. For example, we will sometimes consider a word metric with respect to a subset $S$ which is not necesarily generating; in this case we set $d_S(h,g)=\infty$ when $h^{-1}g\notin \langle S\rangle$. Given two metrics $d_1$ and $d_2$ on a set $X$, we say that $d_1$ is bi-Lipschitz equivalent to $d_2$ (and write $d_1\sim _{Lip}d_2$) if for all $x, y\in X$, $d_1 (x, y)$ is finite if and only if $d_2(x, y)$ is, and the ratios $d_1/d_2$ and $d_2/d_1$ are uniformly bounded on $X\times X$ minus the diagonal.

For a word $W$ in an alphabet $S$, $\|W\|$ denotes its length. For two words $U$ and $V$ we write $U \equiv V$ to denote the letter-by-letter equality between them, and  $U=_GV$ to mean that $U$ and $V$ both represent the same element of $G$. Clearly there is a one to one correspondence between words $W$ in $S$ and paths $p$ in $\Gs$ such that $p_-=1$ and $\lab(p)\equiv W$.

The normal closure of a subset $K\subseteq G$ in a group $G$ (i.e., the minimal normal subgroup of $G$ containing $K$) is denoted by $\ll K\rr$. For  group elements $g$ and $t$, $g^t$ denotes $t^{-1}gt$.  We write $g\sim h$ if $g$ is conjugate to $h$, that is  there exists $t\in G$ such that $g^t=h$. We also say that $g$ and $h$ are \emph{commensurable} if for some $n, k\in\Z\setminus\{0\}$, $g^n\sim h^k$.

A path $p$ in a metric space is called {\it $(\lambda, c)$--quasi--geodesic} for some $\lambda > 0$, $c\ge 0$,  if
$$d(q_-, q_+)\ge \lambda l(q)-c$$ for any subpath $q$ of $p$.

\paragraph{\bf Van Kampen Diagrams.}
Let $G$ be a group given by a presentation
\begin{equation}
G=\langle \mathcal A\; | \; \mathcal O\rangle. \label{ZP}
\end{equation}

Let $\Delta$ be a finite, oriented, connected, simply--connected 2--complex embedded in the plane such that each edge is labeled by an element of $\mathcal A$. We denote the label of an edge $e$ by $\lab(e)$ and require that $\lab(e^{-1})\equiv (\lab (e))^{-1}$. Given a cell $\Pi $ of $\Delta $, we denote by $\partial \Pi$ the boundary of $\Pi $ and $\partial \Delta$ the boundary of $\Delta$. Note that the corresponding labels $\Lab(\partial \Pi $) and $\Lab(\partial \Delta)$ are defined only up to a cyclic permutation. Then $\Delta$ is called a \emph{van Kampen diagram over the presentation (\ref{ZP})} if for each cell $\Pi$ of $\Delta$, there exists $R\in O$ such that $\Lab(\partial \Pi)\equiv R$. For a word $W$ over the alphabet $\mathcal A$, $W=_G1$ if and only if there exists a van Kampen diagram $\Delta $ over (\ref{ZP}) such that $\lab(\partial \Delta
)\equiv W$ \cite[Ch. 5, Theorem 1.1]{LS}.

A geodesic metric space $X$ is called \emph{$\delta$-hyperbolic} if given any geodesic triangle in $X$, each side of the triangle is contained in the union of the closed $\delta$-neighborhoods of the other two sides. It is well-known that a space is hyperbolic if and only if it satisfies a coarse linear isoperimetric inequality. This can be translated to the context of Cayley graphs of groups in the following way.  A group presentation of $G$ of the form (\ref{ZP}) is called \emph{bounded} if $\sup\{\|R\|\;|\; R\in\mathcal O\}<\infty$. Given a van Kampen diagram  $\Delta$ over (\ref{ZP}), let $Area(\Delta)$ denote the number of cells of $\Delta$. Given a word $W$ in $\mathcal A$ with $W=_G1$, we let $Area(W)=\min_{\partial\Delta\equiv W}\{Area(\Delta)\}$, where the minimum is taken over all diagrams with boundary label $W$. The presentation (\ref{ZP}) satisfies a \emph{linear isoperimetric inequality} if there exists a contant $L$ such that for all $W=_G1$, $Area(W)\leq L\|W\|$. The following is well-known and can be easily derived from the results of \cite[Sec. 2, Ch. III.H]{BH}.

\begin{thm}\label{liniso}
Given a generating set $\mathcal A$ of a group $G$, the Cayley graph $\Ga$ is hyperbolic if and only if $G$ has a bounded presentation of the form (\ref{ZP}) which satisfies a linear isoperimetric inequality.
\end{thm}

\paragraph{Acylindrical and WPD actions.}

Recall the definition of an acylindrical action given in the introduction.

\begin{defn}

Let $G$ be a group acting on a metric space $(X, d)$. We say that the action is \emph{acylindrical} if for all $\e>0$ there exist $R>0$ and $N>0$ such that for all $x$, $y\in X$ with $d(x, y)\geq R$, the set 
\[
\{g\in G\;|\; d(x, gx)\leq \e, d(y, gy)\leq \e\}
\]
contains at most $N$ elements.
\end{defn}

Given a group $G$ acting on a hyperbolic metric space $(X, d)$ and $g\in G$, the \emph{translation length} of $g$ is defined as $\tau(g)=\lim_{n\rightarrow \infty}\frac{1}{n}d(x, g^nx)$ for some (equivalently, any) $x\in X$. An element $g\in G$ is called \emph{loxodromic} if $\tau(g)>0$. Equivalently, $g$ is loxodromic if there is an invariant, bi-infinite quasi-geodesic on which $g$ restricts to a non-trivial translation. If $g$ is loxodromic, then the orbit of $g$ has exactly two limit points $\{g^{\pm\infty}\}$ on the boundary $\partial X$. Loxodromic elements $g$ and $h$ are called \emph{independent} if the sets $\{g^{\pm\infty}\}$  and $\{h^{\pm\infty}\}$ are disjoint. $G$ is called \emph{elliptic} if some (equivalently, any) $G$-orbit is bounded. 

 \begin{thm}\cite{Osip}\label{subah}
Suppose $G$ acts acylindrically on a hyperbolic metric space. Then $G$ satisfies exactly one of the following:
\begin{enumerate}
\item $G$ is elliptic.
\item $G$ is virtually cyclic and contains a loxodromic element.
\item $G$ contains infinitely many pairwise independent loxodromic elements.
\end{enumerate}
\end{thm}

Notice that the last condition holds if and only if the action of $G$ is non-elementary. Also, if $G$ acts acylindrically on a hyperbolic metric space, then the induced action of any subgroup $H\leq G$ is acylindrical. Hence this theorem implies that any subgroup of $G$ which is not elliptic or virtually cyclic is acylindrically hyperbolic. 

When this theorem is applied to cyclic groups, it gives the following result of Bowditch.
 
 \begin{lem}\cite{Bow2}\label{loxorell}
Suppose $G$ acts acylindrically on a hyperbolic metric space. Then every element of $G$ is either elliptic or loxodromic.
\end{lem}

In many cases, we will be interested in the action of $G$ on some Cayley graph which we will occasionally need to modify. The next two lemmas show how to do this.

\begin{lem}\label{A1}
Suppose $\tau(h)>0$ with respect to the action of $G$ on $\Gamma(G,\mathcal A_1)$ and $\mathcal A\subseteq \mathcal A_1$ generates $G$. Then $\tau(h)>0$ with respect to the action of $G$ on $\Ga$.
\end{lem}

\begin{proof}
\[
\lim_{n\rightarrow \infty}\frac{1}{n}d_{\mathcal A}(x, h^nx)\geq\lim_{n\rightarrow \infty}\frac{1}{n}d_{\mathcal A_1}(x, h^nx)>0.
\]

\end{proof}

\begin{lem}\label{joinell}
Suppose $\Ga$ is hyperbolic, $G$ acts acylindrically on $\Ga$, and $B\subset G$ is a bounded subset of $\Ga$. Then  $\Gamma(G, \mathcal A\cup B)$ is hyperbolic, the action of $G$ on $\Gamma(G, \mathcal A\cup B)$ is acylindrical, and both actions have the same set of loxodromic elements.
\end{lem}

\begin{proof}
The identity map on $G$ induces a $G$-equivariant quasi-isometry between $\Ga$ and $\Gamma(G, \mathcal A\cup B)$. It follows easily from the definitions that all conditions are preserved under such a map.
\end{proof}

In \cite{BF}, Bestvina-Fujiwara defined a weak form of acylindricity, which they called \emph{weak proper discontinuity} or  \emph{WPD}.

\begin{defn}\cite{BF}\label{defn:WPD}
Let $G$ be a group acting on a hyperbolic metric space $X$ and $h$ a loxodromic element of $G$. We say $h$ satisfies the \emph{WPD condition} (or $h$ is a \emph{WPD element}) if for all $\e>0$ and $x\in X$, there exists $N$ such that 
\begin{equation}\label{WPDeq}
|\{g\in G\;|\;d(x, gx)<\e, d(h^Nx, gh^Nx)<\e\}|<\infty.
\end{equation}
\end{defn}

Note that if $G$ acts acylindrically on a hyperbolic metric space, then every loxodromic element satisfies the WPD condition.

\begin{lem}\cite[Lemma 6.5, Corollary 6.6]{DGO}\label{E(h)}
Let $G$ be a group acting on a hyperbolic metric space $X$, and let $h$ be a loxodromic WPD element. Then $h$ is contained in a unique, maximal elementary subgroup of $G$, called the \emph{elementary closure} of $h$ and denoted $E_G(h)$. Furthermore,  for all $g\in G$, the following are equivalent:
\begin{enumerate}
\item $g\in E_G(h)$.

\item There exists $n\in\N$ such that $g^{-1}h^ng=h^{\pm n}$.

\item There exist $k$, $m\in\Z\setminus\{0\}$ such that $g^{-1}h^kg=h^m$.

\end{enumerate}

Further, for some $r\in\N$,
\[
E_G^+(h):=\{g\in G\;|\;\exists n\in\N, g^{-1}h^ng=h^n\}=C_G(h^r).
\]
\end{lem}

\section{Hyperbolically embedded subgroups}\label{sect:HES}

In \cite{DGO}, Dahmani-Guirardel-Osin introduced the notion of a \emph{hyperbolically embedded subgroup}, which generalizes the peripheral structure of subgroups of relatively hyperbolic groups. Let $G$ be a group, $\Hl $ a collection of subgroups of $G$. Set
\begin{equation}\label{calH}
\mathcal H= \bigsqcup\limits_{\lambda \in \Lambda } H_\lambda.
\end{equation}
Suppose $X\subseteq G$ such that $X\sqcup \mathcal H$ generates $G$. Such an $X$ is called a \emph{relative generating set of $G$ with respect to $\Hl$}. We consider the corresponding Cayley graph $\G$, which may have multiple edges when distinct elements of the disjoint union represent the same element of $G$. Now fix $\lambda\in\Lambda$, and notice that the Cayley graph $\Gamma (H_\lambda, H_\lambda )$ is naturally embedded as a complete subgraph of $\G $. A path $p$ in $\G$ such that $p_-, p_+\in H_\lambda$ is called \emph{admissible} if $p$ contains no edges belonging to $\Gamma (H_\lambda, H_\lambda )$. Note that admissible paths can have edges labeled by elements of $H_\lambda$ as long as the endpoints of these edges do not belong to $H_\lambda$. Given $h,k\in H_\lambda $, let $\dl (h,k)$ be the length of a shortest admissible path from $h$ to $k$, or $\dl(h, k)=\infty$ if no such path exists. $\dl $ is called the \emph{relative metric} on $H_\lambda$. It is convenient to extend the metric $\dl $  the whole group $G$ by assuming $\dl (f,g) :=\dl (1, f^{-1}g)$ if $f^{-1}g\in H_\lambda $ and $\dl (f,g)=\infty $ otherwise. In case the collection consists of a single subgroup $H\le G$, we denote the corresponding relative metric on $H$ simply by $\widehat d$. Recall that a metric space is called \emph{locally finite} if there are finitely many elements inside any ball of finite radius.


\begin{defn}\cite{DGO}\label{hes1}
Let $G$ be a group, $X\subseteq G$. We say that a collection of subgroups $\Hl$ of $G$ is  \emph{\he in $G$ with respect to $X$} if the following conditions hold.
\begin{enumerate}
\item[(a)] $G$ is generated by $X\sqcup \mathcal H$ and the Cayley graph $\G $ is hyperbolic.
\item[(b)] For every $\lambda\in \Lambda $, $(H_\lambda, \dl )$ is a locally finite metric space.
\end{enumerate}
 We write $\Hl\h (G, X)$ to mean that  $\Hl$ is \he in $G$ with respect to $X$ or simply $\Hl\h G$ if we do not need to keep track of the set $X$, that is $\Hl\h (G, X)$ for some $X\subseteq G$. Note that for any group $G$ and any finite subgroup $H$, $H\h (G, G)$. Furthermore for any group $G$, $G\h (G, \emptyset)$. Such cases are are called \emph{degenerate}, and a hyperbolically embedded subgroup $H$ is called \emph{non-degenerate}  whenever $H$ is proper and infinite.
\end{defn}

As with relative hyperbolicity, the notion of hyperbolically embedded subgroups can be expressed in terms of an isoperimetric inequality. Let $G$ be a group, $\Hl$ a collections of subgroups of $G$ and $X\subseteq G$ a relative generating set of $G$ with respect to $\Hl$. Let $\mathcal H$ be defined by (\ref{calH}). Let $F(X)$ be the free group with basis $X$, and consider the free product
\begin{equation}
F=\left( \ast _{\lambda\in \Lambda } H_\lambda  \right) \ast F(X).
\label{F}
\end{equation}
Clearly there is a natural surjective homomorphism $F\twoheadrightarrow G$. If the kernel of this homomorphism is equal to the normal closure of a subset $\mathcal Q\subseteq F$ then we say that $G$ has {\it relative presentation}
\begin{equation}\label{relpres}
\langle X,\; \mathcal H \; |\; \mathcal Q
\rangle .
\end{equation}
The relative presentation (\ref{relpres}) is said to be {\it bounded} if $\sup\{\|R\|\;|\; R\in\mathcal Q\}<\infty$. Furthermore, it is called {\it strongly bounded} if in addition the set of letters from $\mathcal H$ which appear in relators $R\in\mathcal Q$ is finite.

Given a word $W$ in the alphabet $X\sqcup \mathcal H$ such that $W=_G 1$, there exists an expression
\begin{equation}
W=_F\prod\limits_{i=1}^k f_i^{-1}R_i^{\pm 1}f_i \label{prod}
\end{equation}
where $R_i\in \mathcal Q$ and
$f_i\in F $ for $i=1, \ldots , k$. The {\it relative area} of $W$, denoted $Area^{rel}(W)$, is the minimum $k$ such that $W$ has a representation of the form (\ref{prod}).
\begin{thm}\cite[Theorem 4.24]{DGO}\label{AHliniso}
The collection of subgroups $\Hl$ are \emph{hyperbolically embedded in $G$ with respect to $X$} if and only if there exists a strongly bounded relative presentation for $G$ with respect to $X$ and $\Hl$ and there is a constant $L>0$ such that for any word $W$ in $X\sqcup
\mathcal H$ representing the identity in $G$, we have $Area^{rel}
(W)\le L\| W\| $.
\end{thm}

Relative area can also be defined in terms of van Kampen diagrams. Let
$\mathcal S$ denote the set of all words $U$ in the alphabet $\mathcal H$ such that $U=_F1$.
Then $G$ has the ordinary (non--relative) presentation
\begin{equation}\label{Gfull}
G=\langle X\sqcup\mathcal H\; |\;\mathcal S\cup \mathcal Q \rangle .
\end{equation}
Let $\Delta$ be a van Kampen diagram over  (\ref{Gfull}). Let $N_{\mathcal Q}(\Delta)$ denote the number of cells of $\Delta$ whose boundaries are labeled by an element of $\mathcal Q$. Then for any word $W$ in $X\sqcup\mathcal H$ such that $W=_G1$,

$$
Area^{rel}(W)=\min\limits_{\lab (\partial \Delta ) \equiv W} \{
N_\mathcal Q (\Delta )\} ,
$$
where the minimum is taken over all diagrams with
boundary label $W$. Thus, $\Hl\h G$ if $G$ has a strongly bounded presentation with respect to $\Hl$ and all van Kampen diagrams over (\ref{Gfull}) satisfy a linear relative isoperimetric inequality.

In \cite{DGO}, it is shown that many basic properties of relatively hyperbolic groups can be translated to analogous results for groups with hyperbolically embedded subgroups. The following lemmas are examples of this process.

\begin{lem}\cite[Proposition 4.33]{DGO}\label{finint}
Suppose $\Hl\h G$. Then for all $g\in G$, the following hold:
\begin{enumerate}
\item If $g\notin H_\lambda$, then $|H_\lambda\cap H_\lambda^g|<\infty$.

\item If $\lambda\neq \mu$, then $|H_\lambda\cap H_\mu^g|<\infty$.
\end{enumerate}

\end{lem}

\begin{lem}\cite[Corollary 4.27]{DGO}\label{finsymdif}
Let $G$ be a group, $\Hl$ a collection of subgroups, and $X_1$, $X_2\subseteq G$ relative generating sets of $G$  with respect to $\Hl$ such that $|X_1\triangle X_2|<\infty$. Then $\Hl\h (G, X_1)$ if and only if $\Hl\h (G, X_2)$.

\end{lem}

The following two lemmas are simplifications of \cite[Proposition 4.35]{DGO} and \cite[Proposition 4.36]{DGO} respectively.
\begin{lem}\label{hehe}
Suppose $\{H_i\}_{i=1}^n\h G$, and for each $1\leq i\leq n$, $\{K^i_j\}_{j=1}^{m_i}\h H_i$. Then $\{K^i_j\;|\; 1\leq i\leq n, 1\leq j\leq m_i\}\h G$.
\end{lem}

\begin{lem}\label{conjhe}
If $H\h G$, then for any $t\in G$, $H^t\h G$.
\end{lem}

Let $\Hl \h (G,X)$. Let $q$ be a path in the Cayley graph $\G$. An {\it $H_\lambda $-subpath} of $q$ is a non-trivial subpath $p$ such that each edge of $p$ is labeled by an element of $H_\lambda$. An {\it $H_\lambda $-component} of $q$ is a maximal $H_\lambda $-subpath, that is an $H_\lambda $-subpath $p$ such that $p$ is not contained in a longer $H_\lambda $-subpath of $q$ or of any cyclic shift of $q$ if $q$ is a loop. By a {\it component} of $q$ we mean an $H_\lambda $-component of $q$ for some $\lambda \in \Lambda$. If $p$ is an $H_\lambda $-component of some path, then we define the \emph{relative length} of $p$ by $\widehat{\l}_\lambda(p)=\dl(p_-, p_+)$.

Two $H_\lambda $-components $p_1, p_2$ of a path $q$ in $\G $ are called {\it connected} if there exists an edge $e$ such that $\lab(e)\in H_\lambda$ and $e$ connects some vertex of $p_1$ to some vertex of $p_2$. Note that $p_1$ and $p_2$ are connected if and only if all vertices of $p_1$ and $p_2$ belong to the same left coset of $H_\lambda $. A component $p$ of a path $q$ is called \emph{isolated} in $q$ if $p$ is not connected to any other components of $q$.


\begin{lem}\cite[Proposition 4.14]{DGO}\label{C}
Suppose $\Hl\h G$. Then there exists a constant $C$ such that if $\mathcal P=p_1...p_n$ is a geodesic $n$-gon in $\G$  and $I\subseteq\{1,..., n\}$ such that for each $i\in I$, $p_i$ is an isolated $H_{\lambda_i}$ component of $\mathcal P$, then 
\[
\sum_{i\in I}\widehat{\l}_{\lambda_i}(p_i)\le Cn.
\]
\end{lem}

In \cite{DGO}, one of the main sources of examples of groups which contain hyperbolically embedded subgroups is given by elements which satisfy the WPD condition. Recall that group elements $g$ and $h$ are \emph{commensurable} if for some $n, k\in\Z\setminus\{0\}$, $g^n$ is conjugate to $h^k$.

\begin{lem}\cite[Theorem 6.8]{DGO}\label{he-wpd}
Suppose $G$ acts on a hyperbolic metric space $X$ and $h_1$,...,$h_n$, is a collection of non-commensurable loxodromic WPD elements. Then $\{E_G(h_1),...,E_G(h_n)\}\h G$. 
\end{lem}

Given a finitely generated, non-degenerate subgroup $H\h (G, X)$, the next lemma shows explicitly how to find loxodromic, WPD elements with respect to the action of $G$ on $\Gamma(G, X\sqcup H)$.

\begin{lem}\cite[Corollary 6.12]{DGO}\label{lox}
Suppose $H\h (G, X)$ is non-degenerate and finitely generated. Then for all $g\in G\setminus H$, there exist $h_1,...,h_k\in H$ such that $gh_1,...,gh_k$ is a collection of non-commensurable, loxodromic WPD elements with respect to the action of $G$ on $\Gamma(G, X\sqcup H)$. Moreover, if $H$ contains an element of infinite order $h$, then each $h_i$ can be chosen to be a power of $h$.
\end{lem}
\begin{rem}\label{rem:lox}
From the proof of \cite[Corollary 6.12]{DGO}, it is obvious that $h_1$ can be chosen as any element of $H$ such that $\widehat d(1, h_1)$ is sufficiently large. Furthermore, each $h_i$ can be successively chosen as any element of $H$ such that $\widehat d(1, h_i)$ is sufficiently large compared to $\widehat d(1, h_{i-1})$.
\end{rem}

The next theorem is a recent result of Osin which shows that hyperbolically embedded subgroups can be used to build acylindrical actions.

\begin{thm}\cite[Theorem 5.4]{Osip}\label{Ahyp}
Let $G$ be a group, $\Hl$ a finite collection of subgroups of $G$, $X$ a subset of $G$ such that $\Hl\h (G, X)$. Then there exists $Y\subseteq G$ such that $X\subseteq Y$ and the following conditions hold:
\begin{enumerate}
\item $\Hl\h (G, Y)$. In particular, $\Gamma(G, Y\sqcup\mathcal H)$ is hyperbolic.

\item The action of $G$ on $\Gamma(G, Y\sqcup\mathcal H)$ is acylindrical.
\end{enumerate}
\end{thm}

If the subgroups $\Hl$ are non-degenerate, then this action will also be non-elementary \cite[Lemma 5.12]{Osip}. Summarizing the previous results gives the following theorem.

\begin{thm}\label{thmAH}\cite{Osip}
The following are equivalent:
\begin{enumerate}
\item $G\in\X$.
\item $G$ is not virtually cyclic and $G$ admits an action on a hyperbolic metric space such that $G$ contains at least one loxodromic, WPD element.
\item $G$ contains a non-degenerate hyperbolically embedded subgroup.
\item For some generating set $\mathcal A\subseteq G$, $\Ga$ is hyperbolic and the action of $G$ on $\Ga$ is non-elementary and acylindrical.
\end{enumerate}
\end{thm}

In particular, this theorem implies that we can always choose the metric space from the definition of $\X$ to be a Cayley graph of $G$ with respect to some (possibly infinite) generating set.

Note that Lemma \ref{lox} shows how to find $h\in G$ which is a loxodromic, WPD element with respect to the action of $G$ on $\Gamma(G, X\sqcup H)$, and by Lemma \ref{he-wpd} $E_G(h)\h G$. We will show that, in fact,  $E_G(h)\h (G, X\sqcup H)$ (see Corollary \ref{heGX}).

 Instead of working directly with the WPD condition we will use the more general notion of geometrically separated subgroups. The proof in both cases is essentially the same, and we believe the more general statement of Theorem \ref{crit} may be of independent interest. Theorem \ref{crit} is very similar to \cite[Theorem 4.42]{DGO}, however \cite[Theorem 4.42]{DGO} is proven without the assumption that the action is cobounded. By assuming that $G$ is acting on a Cayley graph, we are essentially adding this assumption in order to get an explicit relative generating set. It should be possible to repeat the proof of \cite[Theorem 4.42]{DGO} and keep track of the relative generating set produced there, but this would require quite a bit of technical detail and for our purposes a direct proof is easier.

We will first need a few results about hyperbolic metric spaces. Given a subset $S$ in a geodesic metric space $(X, d)$, we denote by $S^{+\sigma}$ the $\sigma$-neighborhood of $S$. $S$ is called {\it $\sigma $-quasi-convex} if for any two elements $s_1,s_2\in S$, any geodesic in $X$ connecting $s_1$ and $s_2$ belongs to $S^{+\sigma}$.
Let $\mathcal Q=\Qp $ be a collection of subsets of a metric space $X$. One says that $\mathcal Q$ is {\it $t$-dense} for $t\in \mathbb R_+$ if $X$ coincides with the $t$-neighborhood of $\bigcup\mathcal Q$. Further $\mathcal Q$ is {\it quasi-dense} if it is $t$-dense  for some $t\in \mathbb R_+$. Let us fix some positive constant $c$. A \emph{$c$-nerve} of $\mathcal Q$ is a graph with the vertex set $\Pi $ and with $p, q\in \Pi $ adjacent if and only if $d(Q_p,Q_q) \le  c$. Finally we recall that $\mathcal Q$ is {\it uniformly quasi-convex} if there exists $\sigma $ such that $Q_p$ is $\sigma$-quasi-convex for any $p\in \Pi $. The lemma below is an immediate corollary of \cite[Proposition 7.12]{Bow}.

\begin{lem}\label{nerve}
Let $X$ be a hyperbolic space, and let $\mathcal Q=\Qp$ be a quasi-dense collection of uniformly quasi-convex subsets of $X$. Then for any large enough $c$, the $c$-nerve of $\mathcal Q$ is hyperbolic.
\end{lem}

The next lemma is a simplification of \cite[Lemma 25]{Ols92}, see also \cite[Lemma 2.4]{Osi06}. Here two paths $p$ and $q$ are called $\e$-close if either $d(p_-, q_-)\leq \e$ and $d(p_+, q_+)\leq \e$, or if $d(p_-, q_+)\leq \e$ and $d(p_+, q_-)\leq \e$.

\begin{lem}\label{N1N2N3}
Suppose that the set of all sides of a geodesic $n$--gon
$P=p_1p_2\ldots p_n$ in a $\delta $--hyperbolic space is
partitioned into two subsets $A$ and $B$. Let $\rho $
(respectively $\theta $) denote the sum of lengths of sides from
$A$ (respectively $B$). Assume, in addition, that $\theta > \max
\{ \xi n,\, 10^3\rho \} $ for some $\xi \ge 3\delta\cdot 10^4$.
Then there exist two distinct sides $p_i, p_j\in B$ that contain
$13\delta $-close segments of length greater than
$10^{-3}\xi$.
\end{lem}

\begin{defn}\cite{DGO}
Let $G$ be a group acting on a metric space $(X, d)$. A collection of subgroups $\Hl\leq G$ is called \emph{geometrically separated} if for all $\e\geq 0$ and $x\in X$, there exists $R>0$ such that the following holds. Suppose that for some $g\in G$ and some $\lambda$, $\mu\in\Lambda$,
\[
diam(H_\mu(x)\cap(gH_\lambda(x))^{+\e})\geq R.
\]

Then $\lambda=\mu$ and $g\in H_\lambda$.
\end{defn}

\begin{thm}\label{crit}
Let $G$ be a group, $\Hl $ a finite collection of subgroup of $G$. Suppose that the following conditions hold.
\begin{enumerate}
\item[(a)] $G$ is generated by a (possibly infinite) set $X$ such that $\Gamma (G,X)$ is
hyperbolic.

\item[(b)] For every $\lambda \in \Lambda $, $H_\lambda $ is quasi-convex in $\Gamma (G, X)$.

\item[(c)] $\Hl$ is geometrically separated in $\Gamma(G, X)$.

\end{enumerate}
Then the Cayley graph $\G $ is hyperbolic and there exists $C>0$ such that for every $\lambda \in \Lambda $, we have $\dl \sim _{Lip} \dol$, where $\Omega _\lambda =\{ h\in H_\lambda \mid |h|_X\le C\} $.

In particular, if every $H_\lambda $ is locally finite with respect to $d _X$, then $\Hl \h (G,X)$.
\end{thm}

\begin{proof}
Let us first show that the graph $\G $ is hyperbolic. Let $\mathcal Q$ be the collection of all left cosets of subgroups $H_\lambda $, $\lambda \in \Lambda $. We think of $\mathcal Q$ as a collection of subsets of $\Gamma (G,X)$. Since $\Lambda $ is finite and every $H_\lambda $ is quasi-convex in $\Gamma(G,X)$, $\mathcal Q$ is uniformly quasi-convex. Clearly $\mathcal Q$ is quasi-dense. Hence by Lemma \ref{nerve} there exists $c\ge 1$ such that the $c$-nerve of $\mathcal Q$ is hyperbolic. Let $\Sigma $ denote the nerve, and let $\widehat\Gamma $ be the coned-off graph of $G$ with respect to $X$ and $\Hl$. That is, $\widehat\Gamma $ is the graph obtained from $\Gamma(G, X)$ by adding one vertex $v_{gH_\lambda}$ for each left coset of each subgroup $H_\lambda$ and then adding an edge of length $\frac12$ between $v_{gH_\lambda}$ and each vertex of $gH_\lambda$.

Let $d_\Sigma$ and $d_{\widehat\Gamma}$ denote the natural path metrics on $\Sigma$ and $\widehat\Gamma$ respectively. It is easy to see that $\Sigma $ and  $\widehat\Gamma $ are quasi-isometric. Indeed let $\iota \colon V(\Sigma ) \to V(\widehat\Gamma )$ be the map which sends $gH_\lambda\in\mathcal Q$ to $v_{gH_\lambda}$. If $u,v \in V(\Sigma )$ are connected by an edge in $\Sigma $, then there exist elements $g_1, g_2$ of the cosets corresponding to $u$ and $v$ such that $d_X(g_1, g_2)\le c$ in $\Gamma (G,X)$. This implies that $d_{\widehat\Gamma} (\iota (u),\iota (v))\le c+1$. Hence $d_{\widehat\Gamma} (\iota (u), \iota (v))\le (c+1)d_\Sigma (u,v) $ for any $u,v\in V(\Sigma)$. On the other hand, it is straightforward to check that $\iota $ does not decrease the distance. Note that $\iota (V(\Sigma))$ is $1$-dense in $\widehat\Gamma $. Thus $\iota $ extends to a quasi-isometry between $\Sigma $ and $\widehat \Gamma $.

Further observe that $\widehat \Gamma $ is quasi-isometric to $\G $. Indeed the identity map on $G$ induces an isometric embedding $V(\G )\to \widehat\Gamma $ whose image is $1$-dense in $\widehat\Gamma $. Thus $\Sigma $ is quasi-isometric to $\G $ and hence $\G$ is hyperbolic.

Now choose $\sigma$ such that $\mathcal Q$ is $\sigma$-uniformly quasi-convex, fix $\lambda\in\Lambda$ and $h$, $h^\prime\in H_\lambda$. Let $p$ be an admissible path in $\G$ from $h$ to $h^\prime$ such that $\l(p)=\dl(h, h^\prime)$. Let $e$ represent the $H_\lambda$-edge from $h$ to $h^\prime$ in $\G$, and let $c$ be the cycle $pe^{-1}$. Note that $c$ has two types of edges; those labeled by elements of $X$ and those labeled by elements of $\mathcal H$. Now for each edge of $c$ labeled by an element of $\mathcal H$, we can replace this edge with a shortest path in $\Gamma(G, X)$ with the same endpoints. This produces a cycle $c^\prime$ which lives in $\Gamma(G, X)$. We consider $c^\prime=q_1q_2...q_n$ as a geodesic $n$-gon in $\Gamma(G, X)$ where the sides consist of two types:

\begin{enumerate}
\item single edges which represent $X$-edges of $c$.

\item geodesics which represent $\mathcal H$-edges of $c$.
\end{enumerate}

We also suppose the sides of $c^\prime$  are indexed such that $q_n$ is the geodesic which replaced the edge $e^{-1}$. We will first show that $\l(q_n)$  is bounded in terms of $\l(p)$. Partition the sides of $c^\prime$ into two sets  $A$ and $B$, where $A$ consists of sides of the first type and $B$ consists of sides of the second type. As in Lemma \ref{N1N2N3}, let $\rho $ (respectively $\theta $) denote the sum of lengths of sides from $A$ (respectively $B$). Note that $n=\l(c)=\l(p)+1$, $\rho\leq \l(p)$, and $\l(q_n)\leq \theta$. Let $\delta$ be the hyperbolicity constant of $\Gamma(G, X)$ and let $R$ be the constant given by the definition of geometrically separated subgroups for $\e=13\delta+2\sigma$. Choose $\xi=\max\{10^3(R+2\sigma), 3\delta\cdot 10^4\}$.

Suppose $\l(q_n)>  \max\{ \xi n,\, 10^3\rho \} $. Since $\theta\geq\l(q_n)$, we can apply Lemma \ref{N1N2N3} to find two distinct $B$-sides, $q_i$ and $q_j$ of $c^\prime$ which contain $13\delta$-close segments of length at least $10^{-3}\xi\geq R+2\sigma$. This means that there exist vertices $u_1$, $u_2$ on $q_i$ and $v_1$, $v_2$ on $q_j$, and paths $s_1$ and $s_2$ in $\Gamma(G, X)$ such that for $k=1, 2$, we have that $(s_k)_-=u_k$, $(s_k)_+=v_k$, and $\l(s_k)\leq 13\delta$. We assume $i<j$, and let $g=\Lab(q_1...q_{i-1})$ and $g^\prime=\Lab(q_1...q_{j-1})$ if $j<n$ and $g^\prime=1$ otherwise. Then $(q_i)_-$,$(q_i)_+\in gH_\mu$ for some $\mu\in\Lambda$, and thus $q_i$ belongs to the $\sigma$-neighborhood of $gH_\mu$. Similarly, $(q_j)_-$,$(q_j)_+\in g^\prime H_\eta$ for some $\eta\in\Lambda$, and thus $q_j$ belongs to the $\sigma$-neighborhood of $g^\prime H_\eta$. Now for $k=1, 2$, choose vertices $u_k^\prime\in gH_\mu$ such that $d_X(u_k, u_k^\prime)\leq \sigma$ and $v_k^\prime\in g^\prime H_\eta$ such that $d_X(v_k, v_k^\prime)\leq \sigma$. It follows that $d_X(u_k^\prime, v_k^\prime)\leq 13\delta+2\sigma=\e$. Also, $d_X(u_1^\prime, u_2^\prime)\geq (R+2\sigma)-2\sigma=R$. Thus, by the definition of geometric separation, $\mu=\eta$ and $gH_\mu=g^\prime H_\mu$.

Now, let $e_i$, $e_j$ be the $\mathcal H$-edges of $c$ corresponding to $q_i$, $q_j$. We have shown that the vertices of these two edges belong to the same left $H_\mu$ coset; hence, there exists an edge $f$ in $\G$ such that $f_-=(e_i)_-$ and $f_+=(e_j)_+$. If $j<n$, we can replace the subpath of $p$ from $(e_i)_-$ to $(e_j)_+$ by the single edge $f$, resulting in a shorter admissible path from $h$ to $h^\prime$, which contradicts our assumption that $\l(p)=\dl(h, h^\prime)$. If $j=n$, we get that $(e_i)_+, (e_i)_-\in gH_\lambda=g^\prime H_\lambda=H_\lambda$. If $\Lab(e_i)\in H_\lambda$, this violates the definition of an admissible path; however, if $\Lab(e_i)\in H_\mu$ for some $\mu\neq\lambda$, then by geometric separation we get that $\l(q_i)=d_X((e_i)_+, (e_i)_-)\leq R$, contradicting the fact that $\l(q_i)\geq R+2\sigma$. Thus we have contradicted the assumption that $\l(q_n)> \max\{ \xi n,\, 10^3\rho \} $, so we conclude that $\l(q_n)\leq \max\{ \xi n,\, 10^3\rho \} \leq\max\{10^3(R+2\sigma)(\l(p)+1), 3\delta\cdot 10^4(\l(p)+1),\, 10^3\l(p) \}$. Thus $\l(q_n)\leq  D\l(p)$, where $D=\max\{10^3(2R+4\sigma), 6\delta\cdot10^4\}$.

Now denote the vertices of $q_n^{-1}$ by $h=v_0, v_1,...,v_m=h^\prime$. For each $v_i$, we can choose $h_i\in H_\lambda$ such that $d_X(v_i, h_i)\leq\sigma$. It follows that $d_X(h_i, h_{i+1})\leq 2\sigma+1$. Let $C=2\sigma+1$ and define $\Omega _\lambda =\{ h\in H_\lambda \mid |h|_X\le C\} $. Note that
\[
h^{-1}h^\prime=(h^{-1}h_1)(h_1^{-1}h_2)...(h_{m-1}^{-1}h^\prime).
\]

 Since each $h_i^{-1}h_{i+1}\in\Omega_\lambda$, we have that $\dol(h, h^\prime)\leq m=\l(q_n)\leq D\l(p)=D\dl(h, h^\prime)$.  Finally, it is clear that $d_X(h, h^\prime)\leq C\dol(h, h^\prime)$. Since any path labeled only by $X$ is admissible in $\G$, we get that $\dl(h, h^\prime)\leq d_X(h, h^\prime)\leq C\dol(h, h^\prime)$, and thus $\dl \sim _{Lip} \dol$. 

\end{proof}

Our main application of Theorem \ref{crit} is due to the fact that all of the assumptions are satisfied by the elementary closures of a collection of pairwise non-commensurable loxodromic WPD elements; this is shown in the proof of \cite[Theorem 6.8]{DGO}. Thus, we have the following corollary.

\begin{cor}\label{heGX}
Suppose $X$ is a generating set of $G$ such that $\Gamma(G, X)$ is hyperbolic and $\{g_1,...,g_n\}$ is a collection of pairwise non-commensurable loxodromic WPD elements with respect to the action of $G$ on $\Gamma(G, X)$. Then $\{E_G(g_1),...,E_G(g_n)\}\h(G,X)$.
\end{cor}

The following lemma will be useful in Section \ref{HNN&amal} when we are considering HNN-extensions and amalgamated products over cyclic subgroups. In particular, it guarantees that after enlarging the generating set of $G$, we can assume that the associated cyclic subgroups lie in a bounded subset of the corresponding Cayley graph.

\begin{lem}\label{ee}
Let $\Hl\h (G, X)$, and let $a_1,...,a_m\in G$. Then there exists $Y\supseteq X$ such that 
\begin{enumerate}
\item $\Hl\h (G, Y)$.

\item For each $i=1,..,m$, $a_i$ is elliptic with respect to the action of $G$ on $\Gamma(G, Y\sqcup \mathcal H)$.
\end{enumerate}
\end{lem}

\begin{proof}
Since enlarging the generating set does not decrease the set of elliptic elements, it suffices to prove the case when $m=1$ and the general case follows by induction. By Theorem \ref{Ahyp} we can choose a relative generating set $Y_0\supseteq X$ such that $\Hl\h (G, Y_0)$ and $G$ acts acylindrically on $\Gamma(G, Y_0 \sqcup \mathcal H)$. If $a$ is elliptic with respect to this action, we are done. Thus, by Lemma \ref{loxorell} we can assume that $a$ is loxodromic. Since the action is acylindrical, all loxodromic elements satisfy WPD, so by Corollary \ref{heGX}, $E_G(a)\h (G, Y_0\sqcup \mathcal H)$. 

We claim that in fact, $\Hl\h(G, Y_0 \sqcup E_G(a))$. Clearly $\Gamma(G, (Y_0 \sqcup E_G(a))\sqcup \mathcal H)$ is hyperbolic, so we only need to verify that the relative metrics are locally finite. Fix $\lambda\in\Lambda$, $n\in\N$, and $h, h^\prime\in H_\lambda$ such that $\dl(h, h^\prime)\leq n$. Let $p$ be an admissible path between $h$ and $h^\prime$ in $\Gamma(G, (Y_0\sqcup E_G(a))\sqcup \mathcal H)$ such that $\l(p)= \dl(h, h^\prime)$. Let $c$ be the cycle $pe$, where $e$ is the $H_\lambda$-edge from $h^\prime$ to $h$. Suppose $x\in E_G(a)$ is the label of an edge of $p$.

 Now if we consider $E_G(a)$ to be a hyperbolically embedded subgroup (with relative metric $\widehat{d}$) and $c$ as a cycle in the corresponding Cayley graph $\Gamma(G, (Y_0\sqcup\mathcal H)\sqcup E_G(a))$, then $x$ must be isolated in $c$; indeed $e$ is not an $E_G(a)$ component, and $x$ cannot be connected to another component of $p$ since $p$ is the shortest admissible path between $h$ and $h^\prime$. Thus by Lemma \ref{C}, $\widehat{\l}(x)\leq C(n+1)$, where $C$ is the constant from Lemma \ref{C}. Since $E_G(a)$ is locally finite with respect to $\widehat{d}$, the set $\mathcal F_n=\{g\in E_G(a) \;|\; \widehat{d}(1, g)\leq C(n+1)\}$ is finite, and we have shown that $\Lab(x)\in \mathcal F_n$. 

Since $h$ and $h^\prime$ are arbitrary, it follows that if $p$ is any admissible path (with respect to $H_\lambda$) in $\Gamma(G, (Y_0\sqcup E_G(a))\sqcup \mathcal H)$ such that $\l(p)=\dl(p_-, p_+)\leq n$, then the label of each edge of $p$ belongs to the set $Y_0\sqcup \mathcal F_n\sqcup \mathcal H$. It follows that the balls centered at the identity of radius $n$ in both $\Gamma(G, (Y_0\sqcup E_G(a))\sqcup \mathcal H)$ and $\Gamma(G, (Y_0\sqcup \mathcal F_n)\sqcup \mathcal H)$ with respect to the corresponding relative $\dl$-metrics are the same.  Furthermore, by Lemma \ref{finsymdif}, $\Hl\h (G, Y_0\sqcup \mathcal F_n)$, hence these balls contain finitely many elements. Thus, $\Hl\h(G, Y_0\sqcup E_G(a))$. It only remains to set $Y=Y_0\sqcup E_G(a)$; clearly every $\langle a\rangle$-orbit is bounded in $\Gamma(G, Y\sqcup \mathcal H)$. 
\end{proof}


\section{Small cancellation quotients}\label{sectscq}


In this section  we prove various properties of small cancellation quotients. Analogous statements for relatively hyperbolic groups can be found in \cite{Osi10}, and we will refer to \cite{Osi10} for some proofs which work in our case without any changes.  We begin by giving the small cancellation conditions introduced by Olshanskii in \cite{Ols} and also used in \cite{HO1, Osi10}.

We call a set of words $\mathcal R$ {\it symmetrized} if $\mathcal R$ is closed under taking cyclic shifts and inverses. Recall that in classical small cancellation theory, a \emph{piece} is a word which is a common subword of two distinct relators. In the hyperbolic setting, we consider pieces which are ``close" to being common subwords. More precisely,

\begin{defn}\label{piece}
Let $G$ be a group generated by a set $\mathcal A$, $\mathcal R$ a
symmetrized set of words in $\mathcal A$. Let $U$ be a subword of a word $R\in \mathcal R$ and let $\e>0$. $U$ is called an {\it
$\e $--piece}  if there exist a word $R^\prime \in \mathcal R$ and a subword $U^\prime$ of $R^\prime$
such that:

\begin{enumerate}
\item[(1)] $R\equiv UV$, $R^\prime \equiv U^\prime V^\prime $, for
some $V, V^\prime $. 
\item[(2)] $U^\prime =_G YUZ$
for some words $Y,Z$ in $\mathcal A$ satisfying $\max \{ \| Y\| ,
\,\| Z\| \} \le \e $.
 \item[(3)] $YRY^{-1}\ne_G R^\prime $.
\end{enumerate}
Similarly, $U$ is called an {\it
$\e$--primepiece}  if:
\begin{enumerate}
\item[($1^\prime $)] $R\equiv UVU^\prime V^\prime $ for some $V,
U^\prime , V^\prime $.
 \item[($2^\prime $)] $U^\prime =_GYU^{\pm
1}Z$ for some words $Y,Z$ in $\mathcal A$ satisfying $\max\{ \| Y\| ,
\| Z\| \}\le \e $.
\end{enumerate}
\end{defn}
\begin{rem}
$\e$--primepieces are also refered to as $\e^\prime$--pieces in \cite{Ols, Osi10}.
\end{rem}
\begin{defn}\label{SCdef}
The set $\mathcal R$ satisfies the {\it $C(\e , \mu ,
\lambda , c, \rho )$--condition} for some $\e \ge 0$, $\mu >0$,
$\lambda >0$, $c\ge 0$, $\rho >0$, if for any $R\in\mathcal R$,
\begin{enumerate}
\item[(1)] $\| R\| \ge \rho $.
\item[(2)] Any path in the Cayley graph $\Gamma (G, \mathcal A)$
labeled by $R$ is a $(\lambda , c)$--quasi--geodesic.
\item[(3)] For any $\e $--piece $U$ of $R$, $\max \{ \| U\| ,\, \| U^\prime
\| \} < \mu \| R\| $ where $U^\prime$ is defined as in Definition
\ref{piece}.
\end{enumerate}
If in addition condition $(3)$ holds for any $\e$--primepiece of any word $R\in
\mathcal R$, then $\mathcal R$ satisfies the {\it $C_1(\e , \mu ,
\lambda , c, \rho )$--condition}.
\end{defn}

 We will show that for an acylindrically hyperbolic group $G$, the $C(\e , \mu , \lambda , c, \rho )$--condition will be sufficient to guarantee that the corresponding quotient $G/\ll \mathcal R\rr$ is acylindrically hyperbolic (see Lemma \ref{scquot}), while the stronger $C_1(\e , \mu , \lambda , c, \rho )$--condition will be sufficient to ensure that no new torsion is created in the quotient (see Lemma \ref{tor}). 



Fix a group $G$ and suppose $\Hl\h (G, X)$. By Theorem \ref{AHliniso}, there exists a constant $L$ such that $G$ has a strongly bounded relative presentation $\langle X,\; \mathcal H \; |\; \mathcal Q\rangle$ which satisfies $Area^{rel}(W)\leq L\|W\|$ for any word $W$ in $X\sqcup \mathcal H$ equal to the identity in $G$. Set $\mathcal A=X\sqcup \mathcal H$ and $\mathcal O=\mathcal S\cup \mathcal Q$, where $\mathcal S$ is defined as the set of relators in each $H_\lambda$ as in equation (\ref{Gfull}) of Section \ref{sect:HES}. Hence $G$ is given by the presentation

\begin{equation}\label{GA0}
G=\langle \mathcal A\; | \; \mathcal O\rangle .
\end{equation}

Given a set of words $\mathcal R$, let $\overline{G}$ denote the quotient of $G$ given by the presentation 
\begin{equation}\label{quot}
\overline{G}= \langle \mathcal A\; |\;
\mathcal O\cup \mathcal R\rangle .
\end{equation}

\begin{lem}\label{scquot}
Let $G$ and $\overline{G}$ be defined by (\ref{GA0}) and (\ref{quot}) respectively. For any $\lambda\in (0, 1]$, $c\geq 0$, $N>0$, there exist $\mu>0$, $\e>0$, and $\rho>0$ such that for any strongly bounded symmetrized set of words $\mathcal R$ satisfying the $C(\e, \mu, 
\lambda, c, \rho)$-condition, the following hold.
\begin{enumerate}
\item The restriction of the natural homomorphism $\gamma\colon G\to \overline{G}$ to $B_{\mathcal A}(N)$ is injective. In particular, $\gamma|_{\bigcup_{\lambda\in\Lambda} H_\lambda}$ is injective. 
\item $\{\gamma(H_{\lambda})\}_{\lambda\in\Lambda}\h \overline{G}$.

\end{enumerate}
\end{lem}
\begin{proof}

Clearly $\overline{G}$ is given by the strongly bounded relative presentation  $\langle X,\; \mathcal H \; |\; \mathcal Q \cup \mathcal R\rangle$. Hence by Theorem \ref{AHliniso}, to show (2) it suffices to show that all van Kampen diagrams over this presentation satisfy a linear relative isoperimeric inequality. The proof of this and condition (1) is exactly the same as \cite[Lemma 5.1]{Osi10}. 

\end{proof}

Note that $\Ga$ is hyperbolic by the definition of $\mathcal A$. For the remainder of this section, we assume in addition that the action of $G$ on $\Ga$ is acylindrical. This can be done without loss of generality by Theorem \ref{Ahyp}. Recall that $\tau(g)$ denotes the translation length of the element $g$.

\begin{lem}\cite[Lemma 2.2]{Bow2}\label{cyc}
Suppose $G$ acts acylindrically on a hyperbolic metric space. Then there exists $d>0$ such that for all loxodromic elements $g$, $\tau(g)\geq d$.
\end{lem}

A path $p$ in a metric space is called a \emph{$k$-local geodesic} if any subpath of $p$ of length at most $k$ is geodesic.

\begin{lem}\cite[Ch. III.H, Theorem 1.13]{BH}\label{lgqg}
Let $p$ be a $k$-local geodesic in a $\delta$-hyperbolic metric space for some $k>8\delta$. Then $p$  is a $(\frac{1}{3}, 2\delta)$-quasi-geodesic.
\end{lem}

\begin{lem}\label{nolongell}
Suppose $g$ is the shortest element in its conjugacy class and $|g|_{\mathcal A}> 8\delta$. Let $W$ be a word in $\mathcal A$ representing $g$ such that $\|W\|=|g|_{\mathcal A}$. Then for all $n\in \N$, any path in $\Ga$ labeled by $W^n$ is a $(\frac{1}{3}, 2\delta)$ quasi-geodesic. In particular, $g$ is loxodromic.
\end{lem}

\begin{proof}
First, since no cyclic shift of $W$ can have shorter length than $W$, any path $p$ labeled by $W^n$ is a $k$-local geodesic where $k>8\delta$, and hence $W^n$ is a $(\frac{1}{3}, 2\delta)$ quasi-geodesic by Lemma \ref{lgqg}.
\end{proof}

\begin{lem}\label{aaqg}
There exist $\alpha$ and $a$ such that the following holds: Let $g$ be loxodromic and the shortest element in its conjugacy class, and let $W$ be a word in $\mathcal A$ representing $g$ such that $\|W\|=|g|_{\mathcal A}$. Then for all $n\in \N$, any path in $\Ga$ labeled by $W^n$ is a $(\alpha, a)$ quasi-geodesic.
\end{lem}
\begin{proof}
If $|g|_{\mathcal A}> 8\delta$, then Lemma \ref{nolongell} shows that $W^n$ is a $(\frac{1}{3}, 2\delta)$ quasi-geodesic. Now suppose $|g|_{\mathcal A}\leq 8\delta$. Let $d$ be the constant provided by Lemma \ref{cyc}. Then
\[
|g^n|_{\mathcal A}\geq n\inf_i\left(\frac{1}{i}|g^i|_{\mathcal A}\right)\geq nd\geq\frac{d}{8\delta}n|g|_{\mathcal A}.
\]

Thus, any path labeled by $W^n$ is a $(\frac{d}{8\delta}, 8\delta)$ quasi-geodesic. Thus we can set $\alpha=\min\{\frac{1}{3}, \frac{d}{8\delta}\}$ and $a=8\delta$.
\end{proof}

\begin{lem}\label{tor}
Let $G$ and $\overline{G}$ be defined by (\ref{GA0}) and (\ref{quot}) respectively. For any $\lambda\in(0,1]$, $c\geq 0$ there are $\mu>0$, $\e> 0$, and $\rho>0$ such that the following condition holds. Suppose that $\mathcal R$ is a symmetrized set of words in $\mathcal A$ satisfying the $C_1(\e, \mu, \lambda, c, \rho)$-condition. Then every element of $\overline{G}$ of order $n$ is the image of an element of $G$ of order $n$.

\end{lem}

\begin{proof}
Let $\alpha$ and $a$ be the constants from Lemma \ref{aaqg}. Note that it suffices to assume $\lambda< \alpha$ and $c>a$, as the $C_1(\e, \mu, \lambda, c, \rho)$-condition becomes stronger as $\lambda$ increases and $c$ decreases. Now we can choose $\mu$, $\e$, and $\rho$ satisfying the conditions of Lemma \ref{scquot} with $N=8\delta+1$. Now suppose $\bar{g}\in \overline{G}$ has order $n$.  Without loss of generality we assume that $\bar{g}$ is the shortest element of its conjugacy class. Let $W$ be a shortest word in $\mathcal A$ representing $\bar{g}$ in $\overline{G}$, and let $g$ be the preimage of $\bar{g}$ represented by $W$. Suppose towards a contradiction that $g^n\neq 1$.

Suppose first that $g$ is elliptic. Then $g^n$ is elliptic, and hence $g^n$ is conjugate to an element $h$ where $|h|_{\mathcal A}\leq 8\delta$ by Lemma \ref{nolongell}. Then $h\neq 1$ but the image of $h$ in $\overline{G}$ is $1$, which contradicts the first condition of Lemma \ref{scquot}.

Thus, we can assume that $g$ is loxodromic, and hence any path labeled by $W^n$ is a $(\lambda, c)$ quasi-geodesic by Lemma \ref{aaqg}. If $\Delta$ is a diagram over (\ref{quot}) with boundary label $W^n$, then $\Delta$ must contain $\mathcal R$-cells since $g^n\neq 1$.  Now for sufficiently small $\mu$ and sufficiently large $\e$ and $\rho$, $\Delta$ must contain an $\mathcal R$-cell whose boundary label will violate the $C_1(\e, \mu, \lambda, c, \rho)$ condition; the proof of this is identical to the proof of \cite[Lemma 6.3]{Osi10}.

\end{proof}

\section{Small cancellation words and suitable subgroups}\label{sect:5}

Let $\Hl\h (G, X)$. We will consider words $W$ in $X\sqcup\mathcal H$ which satisfy the following conditions given in \cite{DGO}: 
\begin{enumerate}

\item[(W1)] $W$ contains no subwords of the form $xy$ where $x$, $y\in X$.

\item[(W2)] If $W$ contains $h\in H_\lambda$ for some $\lambda\in\Lambda$, then  $\dl(1, h)\geq 50C$, where $C$ is the constant from Lemma \ref{C}. In particular, this implies that $h^{\pm 1}\neq_G x$ for any $x\in X$.

\item[(W3)] If $W$ contains a subword $h_1xh_2$ (respectively, $h_1h_2$) where $x\in X$, $h_1\in H_\lambda$ and $h_2\in H_\mu$, then either $\lambda\neq\mu$ or the element of $G$ represented by  $x$ does not belong to $H_\lambda$ (respectively, $\lambda\neq \mu$).
\end{enumerate}

Paths $p$ and $q$ are called \emph{oriented $\e$-close} if $d(p_-, q_-)\leq\e$ and $d(p_+, q_+)\leq \e$.

\begin{lem}\cite[Lemma 4.21]{DGO}\label{conscomp}

\begin{enumerate}
\item If $p$ is a path in $\G$ labeled by a word which satisfies $(W1)-(W3)$, then $p$ is a $\left(\frac{1}{4},1\right)$ quasi-geodesic. 

\item For all $\e>0$ and $k\in\N$, there exists  a constant $M=M(\e, k)$ such that if $p$ and $q$ are oriented $\e$-close paths in $\G$ whose labels satisfy $(W1)-(W3)$ and $\l(p)\geq M$, then at least $k$ consecutive components of $p$ are connected to consecutive components of $q$.
\end{enumerate}
\end{lem}

We will also consider words $W$ which satisfy
\begin{enumerate}
\item[(W4)] There exists $\alpha, \beta\in\Lambda$ such that $H_\alpha\cap H_\beta=\{1\}$ and $W\equiv U_1xU_2$, where $U_1$, $U_2$ are (possibly empty) words in $H_\alpha\sqcup H_\beta$ and $x\in X\cup\{1\}$. 
\end{enumerate}

\begin{lem}\label{commonedge}
Let $\e>0$ and let $M=M(\e, 9)$ be the constant from Lemma \ref{conscomp}. Suppose $p$ and $q$ are oriented $\e$-close paths in $\G$ which are labeled by words which satisfy $(W1)-(W4)$. If $\l(p)\geq M$, then $p$ and $q$ have a common edge.
\end{lem}

\begin{proof}
By Lemma \ref{conscomp}, $p$ has at least $9$ consecutive components connected to consecutive components of $q$. In general, consecutive components may be separated by edges whose label belongs to $X$. However, since there is at most one edge of $p$ whose label belongs to $X$ by $(W4)$, at least $5$ of these components will form a connected subpath of $p$. Considering the corresponding 5 components on $q$ and applying $(W4)$ in the same way, we get that at least $3$ of these components must form a connected subpath of $q$. Hence $p=p_1u_1u_2u_3p_2$ and $q=q_1v_1v_2v_3q_2$, where each $u_i$ is a component of $p$ connected to the component $v_i$ of $q$ (note that each component consists of a single edge by $(W3)$). Without loss of generality we assume that $u_1$ and $u_3$ are $H_\alpha$-components and $u_2$ is an $H_\beta$-component. Now if $e$ is an edge from $(u_1)_+=(u_2)_-$ to $(v_1)_+=(v_2)_-$, then $\Lab(e)\in H_\alpha \cap H_\beta=\{1\}$. Thus, these vertices actually coincide, that is $(u_2)_-=(v_2)_-$. Similarly, $(u_2)_+=(v_2)_+$, and since there is a unique edge labeled by an element of $H_\beta$ between these vertices, we have that $u_2=v_2$.
\end{proof}

\begin{prop}\label{scwords}
Fix any $\e>0$ and suppose $W\equiv xa_1..a_n$ satisfies $(W1)-(W4)$, where $x\in X\cup\{1\}$ and each $a_i\in H_\alpha\sqcup H_\beta$. Suppose, in addition, $a_1^{\pm 1},...,a_n^{\pm 1}$ are all distinct elements of $G$. Let $M=M(\e, 9)$ be the constant given by Lemma \ref{conscomp}. Then the set $\mathcal R$ of all cyclic shifts of $W^{\pm 1}$ satisfies the $C_1(\e,\frac{M}{n},\frac{1}{4},1, n)$-condition.
\end{prop}

\begin{proof}
The proof is similar to the proof of \cite[Theorem 7.5]{Osi10}. Clearly $\mathcal R$ satisfies the first condition of Definition \ref{SCdef}. Lemma \ref{conscomp} gives that $\mathcal R$ satisfies the second condition of Definition \ref{SCdef}. Now suppose $U$ is an $\e$-piece of some $R\in\mathcal R$. In the notation of Definition \ref{piece}, we assume without loss of generality that $\|U\|=\max\{\|U\|, \|U^\prime\|\}$. Assume
\begin{equation}\label{bigpiece}
\|U\|\geq\frac{M}{n}\|R\|\geq M.
\end{equation} 

By the definition of an $\e$-piece, there are oriented $\e$-close paths $p$ and $q$ in $\G$  such that $\Lab(p)\equiv U$, $\Lab(q)\equiv U^\prime$. (\ref{bigpiece}) gives that $p$ and $q$ satisfy the conditions of Lemma \ref{commonedge}, and thus $p$ and $q$ share a common edge $e$. Thus, we can decompose $p=p_1ep_2$ and $q=q_1eq_2$; let $U_1\Lab(e)U_2$ be the corresponding decomposition of $U$ and $U_1^\prime\Lab(e)U_2^\prime$ the corresponding decomposition of $U^\prime$. Let $s$ be a path from $q_-$ to $p_-$ such that $\l(s)\leq\e$, and let  $Y=\Lab(s)$. Then 
\[
R\equiv U_1\Lab(e)U_2V
\] 
and 
\[
R^\prime\equiv U_1^\prime\Lab(e)U_2^\prime V^\prime.
\] 

Since $\Lab(e)$ only appears once in $W^{\pm 1}$, we have that $R$ and $R^\prime$ are cyclic shifts of the same word and
\[
U_2VU_1\equiv U_2^\prime V^\prime U_1^\prime.
\]

Also $Y=_GU_1^\prime U_1^{-1}$ since this labels the cycle $sp_1q_1^{-1}$. Thus,
\[
YRY^{-1}=_GU_1^\prime U_1^{-1}U_1\Lab(e)U_2VU_1(U_1^\prime)^{-1}=_GU_1^\prime\Lab(e)U_2^\prime V^\prime=_GR^\prime
\]

which contradicts the definition of a  $\e$-piece. 

Similarly, if $U$ is an $\e$-primepiece, then $R\equiv UVU^\prime V^\prime$, and arguing as above we get that $U$ and $U^\prime$ share a common letter from $X\sqcup \mathcal H$. However each letter $a\in X\sqcup\mathcal H$ appears at most once in $R$, and if $a$ appears then $a^{-1}$ does not.
\end{proof}

\paragraph{Suitable subgroups.}

Our goal now will be to describe the structure of suitable subgroups. As we will see, it is this structure which allows us to find words satisfying the conditions of Proposition \ref{scwords} with respect to an appropriate generating set. 

Fix $\mathcal A\subset G$ such that $\Ga$ is hyperbolic and $G$ acts acylindrically on $\Ga$. For the rest of this section, unless otherwise stated a subgroup will be called non-elementary if it is non-elementary with respect to the action of $G$ on $\Ga$. Similarly, an element will be called loxodromic if it is loxodromic with respect to this action. In particular, all loxodromic elements will satisfy WPD. 

\begin{lem}\label{nonelsub}
Suppose $S$ is a non-elementary subgroup of $G$. Then for all $k\geq 1$, $S$ contains pairwise non-commensurable loxodromic elements $f_1,..., f_k$, such that $E_G(f_i)=E_G^+(f_i)$.
\end{lem}
\begin{proof}
We will basically follow the proof of \cite[Lemma 6.16]{DGO}. By Theorem \ref{subah}, since $S$ is non-elementary, it contains a loxodromic element $h$, and an element $g$ such that $g\notin E_G(h)$. By Lemma \ref{lox}, for sufficiently large $n_1, n_2, n_3$, $gh^{n_1}, gh^{n_2}, gh^{n_3}$ are pairwise non-commensurable loxodromic elements with respect to $\Gamma(G, \mathcal A\sqcup E_G(h))$, and by Lemma \ref{A1} these elements are loxodromic with respect to $\Ga$. Thus, letting $H_i=E_G(gh^{n_i})$, we get that $\{H_1, H_2 ,H_3\}\h (G, \mathcal A)$ by Corollary \ref{heGX}. Now we can choose $a\in H_1\cap S$, $b\in H_2\cap S$ which satisfy $\widehat{d}_1(1, a)\geq 50C$ and  $\widehat{d}_2(1, b)\geq 50C$, where $C$ is the constant given by Lemma \ref{C}. Then $ab$ cannot belong to $H_3$ by Lemma \ref{C}, so by Lemma \ref{lox} we can find $c_1$,...,$c_k\in H_3\cap S$ such that $\widehat{d}_3(1, c_i)\geq 50C$ and the elements $f_i=abc_i$ are non-commensurable, loxodromic WPD elements with respect to the action of $G$ on $\Gamma(G, \mathcal A_1)$, where $\mathcal A_1=\mathcal A \sqcup H_1\sqcup H_2\sqcup H_3$. Next we will show that $E_G(f_i)=E_G^+(f_i)$. Suppose that $t\in E_G(f_i)$. Then for some $n\in\N$, $t^{-1}f_i^nt=f_i^{\pm n}$. Let $\e=|t|_{\mathcal A_1}$. Then there are oriented $\e$-close paths $p$ and $q$ labeled by $(abc_i)^n$ and $(abc_i)^{\pm n}$. Passing to a multiple of $n$, we can assume that $n\geq M$ where $M=M(\e, 2)$ is the constant provided by Lemma \ref{conscomp}. Then the labels of $p$ and $q$ satisfy $(W1)$-$(W3)$, so we can apply Lemma \ref{conscomp} to get that $p$ and $q$ have two consecutive components. But then the label of $q$ must be $(abc_i)^n$, because the sequences 123123... and 321321... have no common subsequences of length 2. Thus, $t^{-1}f_i^nt=f_i^{n}$, hence $t\in E_G^+(f_i)$. Finally, note that each $f_i$ is loxodromic with respect to the action of $G$ on $\Ga$ by Lemma \ref{A1}.
\end{proof}

Now given a subgroup $S\leq G$, let $\mathcal L_S=\{h\in S\;|\; h\text{ is loxodromic and }E_G(h)=E_G^+(h)\}$. Now define $K_G(S)$ by
\[
K_G(S)=\bigcap_{h\in\mathcal L_S}E_G(h).
\]
 
 The following lemma shows that $K_G(S)$ can be defined independently of $\Ga$.
\begin{lem}\label{K(S)}
Let $S$ be a non-elementary subgroup of $G$. Then $K_G(S)$ is the maximal finite subgroup of $G$ normalized by $S$. In addition, for any infinite subgroup $H\leq S$ such that $H\h G$, $K_G(S)\leq H$.
\end{lem}

\begin{proof}
By Lemma \ref{nonelsub}, $\mathcal L_S$ contains non-commensurable elements $f_1$ and $f_2$. Then by Lemma \ref{finint} $K_G(S)\subseteq E_G(f_1)\cap E_G(f_2)$ is finite. $K_G(S)$ is normalized by $S$ as the set $\mathcal L_S$ is invariant under conjugation by $S$ and for each $g\in S$, $h\in\mathcal L_S$, $E_G(g^{-1}hg)=g^{-1}E_G(h)g$. Now suppose $N$ is a finite subgroup of $G$ such that for all $g\in S$, $g^{-1}Ng=N$. Then for each $h\in\mathcal L_S$, there exists $n$ such that $N\leq C_G(h^n)$, and thus $N\leq E_G(h)$ for all $h\in\mathcal L_S$.

Suppose now that $H\leq S$ and $H\h G$. Then a finite-index subgroup of $H$ centralizes $K_S(G)$, and hence $K_S(G)\leq H$ by Lemma \ref{finint}.
\end{proof}

In \cite{DGO}, it is shown that every $G\in\X$ contains a maximal finite normal subgroup, called the \emph{finite radical} of $G$ and denoted by $K(G)$. In our notation, $K(G)$ is the same as $K_G(G)$. Now if $S$ is a non-elementary subgroup of $G\in\X$, then $S\in\X$, so $S$ has a finite radical $K(S)$. Clearly $K(G)\cap S\leq K(S)\leq K_G(S)$, but in general none of the reverse inclusions hold. Indeed suppose $S\in\X$ with $K(S)\neq\{1\}$. Let $G=(S\times A)\ast H$, where $A$ is finite and $H$ is non-trivial. Then $K(G)=\{1\}$ and $K_G(S)=K(S)\times A$.

\begin{lem}\label{yi}
Let $S$ be a non-elementary subgroup of $G$. Then we can find non-commensurable, loxodromic elements $h_1,...,h_m$ such that $E_G(h_i)=\langle h_i\rangle\times K_G(S)$.
\end{lem}

\begin{proof}
First, since $K_G(S)$ is finite, we can find non-commensurable elements $f_1$,...,$f_k\in\mathcal L_S$ such that $K_G(S)=E_G(f_1)\cap...\cap E_G(f_k)$, and we can further assume that $k\geq 3$. By Lemma \ref{heGX}, $\{E_G(f_1),...,E_G(f_k)\}\h(G,\mathcal A)$. Let $\mathcal A_1=\mathcal A\sqcup E_G(f_1)\sqcup...\sqcup E_G(f_k)$, and consider the action of $G$ on $\Gamma(G,\mathcal A_1)$. For each $1\leq i\leq k$ set $a_i=f_i^{n_i}$  where $n_i$ is chosen such that
\begin{enumerate}
\item $E_G(f_i)=C_G(a_i)$.
\item $\widehat{d}_i(1, a_i)\geq 50C$. 
\item $h=a_1...a_k$ is a loxodromic WPD element with respect to $\Gamma(G, \mathcal A_1)$.

\end{enumerate}

(Here $\widehat{d}_i$ denotes the relative metric on $ E_G(f_i)$). The first condition can be ensured for each $a_i$ by Lemma \ref{E(h)}. Passing to a sufficiently high multiple of an exponent which satisfies the first condition gives an exponent which satisfies the first two conditions. We now fix $n_1,...,n_{k-1}$ such that the corresponding $a_1,...,a_{k-1}$ satisfy the first two conditions. Then $a_1...a_{k-1}\notin E_G(f_k)$ by Lemma \ref{C}, so by Lemma \ref{lox} and Remark \ref{rem:lox}, we can choose $n_k$ a sufficiently high multiple of an exponent which satisfies the first two conditions such that all three conditions are satisfied. We will show that, in fact, $E_G(h)=\langle h\rangle\times K_S(G)$. Let $t\in E_G(h)$, and let $\e=|t|_{\mathcal A_1}$. Then by Lemma \ref{E(h)}, there exists $n$ such that  
\begin{equation}\label{ygn}
t^{-1}h^nt=h^{\pm n}. 
\end{equation}
Up to passing to a multiple of $n$, we can assume that
\[
n\geq \frac{M}{k}
\]
where $M=M(\e, k)$ is the constant provided by Lemma \ref{conscomp}. Now (\ref{ygn}) gives that there are oriented $\e$-close paths $p$ and $q$ in $\Gamma(G, \mathcal A_1)$, such that $p$ is labeled by $(a_1...a_k)^n$ and $q$ is labeled by $(a_1...a_k)^{\pm n}$; notice that the labels of these paths satisfy the conditions $(W1)-(W3)$. Furthermore, there is a path $r$ connecting $p_-$ to $q_-$ such that $\Lab(r)=t$. Now we can apply Lemma \ref{conscomp} to get $k$ consecutive components of $p$ connected to consecutive components of $q$. As in the proof of Lemma \ref{nonelsub}, this gives that $q$ is labeled by $(a_1...a_k)^n$ (not $(a_1...a_k)^{-n}$) since $k\geq 3$. Let $p=p_0u_1...u_kp_1$ and $q=q_0v_1...v_kq_1$, where each $u_i$ is a component of $p$ connected to the component $v_i$ of $q$ (see Figure \ref{fig}). Let $e_0$ be the edge which connects $(u_1)_-$ and $(v_1)_-$, and let $e_i$ be the edge which connects $(u_i)_+$ to $(v_i)_+$. Let $c=\Lab(e_0)$.

\begin{figure}
\centering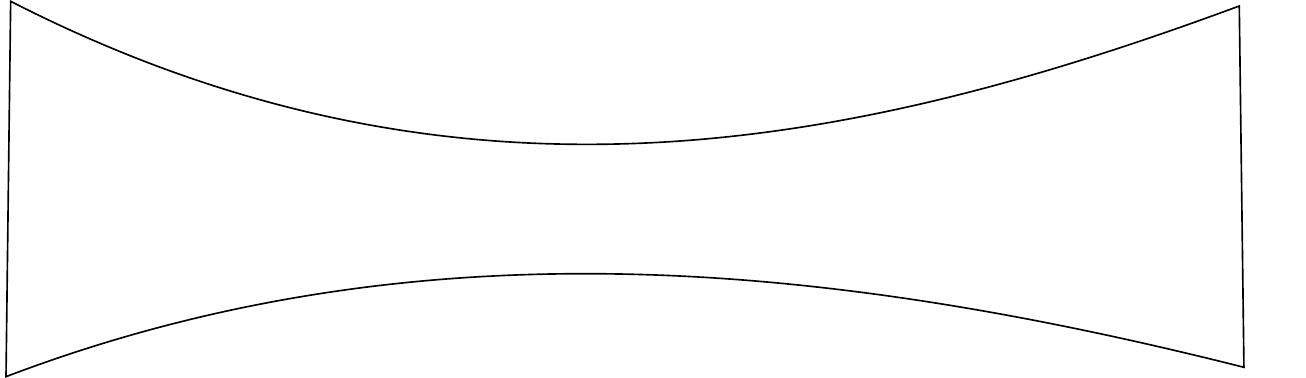
\caption{}\label{fig}
\end{figure}

Since $E_G(f_i)=C_G(a_i)$ for each $1\leq i\leq k$, we get that $c$ commutes with $\Lab(u_1)=\Lab(v_1)$. Thus, $c=\Lab(e_1)$, and repeating this argument we get that $c=\Lab(e_i)$ for each $0\leq i\leq k$. Thus, $c\in E_G(f_1)\cap...\cap E_G(f_k)=K_G(S)$. Now observe that $\Lab(p_0)=(a_1...a_k)^la_1....a_j$ and $\Lab(q_0)=(a_1...a_k)^ma_1....a_j$ for some $m, l\in\N\cup\{0\}$ and $0\leq j\leq k$. Now $rq_0e_0^{-1}p_0^{-1}$ is a cycle in $\Gamma(G, \mathcal A_1)$ and $c$ commutes with each $a_i$, so we get that
 \[
 t=(a_1...a_k)^la_1....a_jca_j^{-1}...a_1^{-1}(a_1...a_k)^{-m}=h^{l-m}c.
 \]
Thus, we have shown that $E_G(h)=\langle h\rangle K_G(S)$. Finally, note that all elements of $K_G(S)$ commute with each $a_i$ and hence commute with $h$. Therefore, $E_G(h)=\langle h\rangle\times K_G(S)$. Now if we set $h_i=a_1...a_k^{l_i}$ for sufficiently large $l_i$, the elements $h_1...h_m$ will all be non-commensurable, loxodromic, WPD elements with respect to $\Gamma(G, \mathcal A_1)$ by Lemma \ref{lox} and Remark \ref{rem:lox}, and the same proof will show that each $h_i$ will satisfy $E_G(h_i)=\langle h_i\rangle\times K_G(S)$. It only remains to note that these elements are all loxodromic with respect to $\Ga$ by Lemma \ref{A1}.

\end{proof}

Recall that a subgroup $S$ of $G$ which is non-elementary (with respect to $\Ga$) is called \emph{suitable} (with respect to $\mathcal A$) if $S$ does not normalize any finite subgroups of $G$. By Lemma \ref{K(S)}, $S$ is suitable if and only if $K_G(S)=\{1\}$. The next two results characterize suitable subgroups by the cyclic hyperbolically embedded subgroups they contain. The first is an immediate corollary of Lemma \ref{yi} and Corollary \ref{heGX}.

\begin{cor}\label {suitsubc}
Suppose $S$ is suitable with respect to $\mathcal A$. Then for all $k\in\N$, $S$ contains non-commensurable, loxodromic elements $h_1,...,h_k$ such that $E_G(h_i)=\langle h_i\rangle$ for $i=1,...,k$. In particular, $\{\langle h_1\rangle,...,\langle h_k\rangle\}\h (G,\mathcal A)$. 
\end{cor}

\begin{lem}\label{suitsubl}
If $S$ contains an infinite order element $h$ such that $\langle h\rangle$ is a proper subgroup of $S$ and $\langle h\rangle\h (G, X)$, then $S$ is suitable with respect to $\mathcal A$ for some $\mathcal A\supseteq X$.
\end{lem}

 \begin{proof}
Note that  Lemma \ref{finint} gives that $\langle h\rangle$ does not have finite index in any subgroup of $G$, so $S$ is not virtually cyclic. By Theorem \ref{Ahyp}, there exists $X\subseteq Y\subseteq G$ such that $\langle h\rangle\h (G, Y)$ and the action of $G$ on $\Gamma(G, Y\sqcup\langle h\rangle)$ is acylindrical; set $\mathcal A=Y\sqcup\langle h\rangle$. Now if $g\in S\setminus \langle h\rangle$, then there exists $n\in\N$ such that $gh^n$ is loxodromic with respect to $\Gamma(G, \mathcal A)$  by Lemma \ref{lox}. Since $S$ is not virtually cyclic, the action of $S$ on $\Gamma(G, \mathcal A)$ is non-elementary by Theorem \ref{subah}. Finally, by Lemma \ref{K(S)}, $K_G(S)$ is a finite subgroup of $\langle h\rangle$, thus $K_G(S)=\{1\}$.
 
 \end{proof}

The next lemma follows from Lemma \ref{suitsubl} and Lemma \ref{hehe}.

\begin{lem}\label{suithe}
Suppose $H\in\X$, $S$ is a suitable subgroup of $H$, and $H\h G$. Then $S$ is a suitable subgroup of $G$. 
\end{lem}

Notice that a group $G\in\X$ will contain suitable subgroups if and only if $K(G)=\{1\}$. However, the following lemma shows that for most purposes this is a minor obstruction; recall that $\X_0$ denotes the class of $G\in\X$ such that $G$ has no finite normal subgroups, or equivalently $K(G)=\{1\}$.

\begin{lem}\label{k(g)}
Let $G\in\X$. Then $G/K(G)\in\X_0$.
\end{lem}

\begin{proof}
By Theorem \ref{thmAH}, we can assume that for some generating set $\mathcal A$, $\Ga$ is hyperbolic and the action of $G$ on $\Ga$ is acylindrical. By Lemma \ref{joinell}, we can assume that $K(G)\subseteq \mathcal A$. Let $G^\prime=G/K(G)$, and let $\mathcal A^\prime$ be the image of $\mathcal A$ in $G^\prime$. Let $x^\prime, y^\prime\in G^\prime$, and let $x, y\in G$ be preimages of $x^\prime$ and $y^\prime$ respectively. Clearly $d_{\mathcal A}(x, y)\geq d_{\mathcal A^\prime}(x^\prime, y^\prime)$. Also, for some $k\in K(G)$, $d_{\mathcal A^\prime}(x^\prime, y^\prime)=d_{\mathcal A}(x, yk)\geq d_{\mathcal A}(x, y)-1$ since $K(G)\subseteq \mathcal A$. Thus,
\[
d_{\mathcal A}(x, y)\geq d_{\mathcal A^\prime}(x^\prime, y^\prime)\geq d_{\mathcal A}(x, y)-1.
\]

Combining this with the fact that each element of $G^\prime$ has finitely many preimages in $G$, it is easy to see that acylindricity of the action of $G$ on $\Ga$ implies that the action of $G^\prime$ on $\Gamma(G^\prime, \mathcal A^\prime)$ is acylindrical. Since $K(G)$ is finite, these spaces are quasi-isometric, hence $\Gamma(G^\prime, \mathcal A^\prime)$ is hyperbolic and non-elementary. Finally maximality of $K(G)$ gives that $G^\prime$ contains no finite normal subgroups, thus $G^\prime\in\X_0$.

\end{proof}

\section{Suitable subgroups of HNN-extensions and amalgamated products}\label{HNN&amal}

In this section we will show that suitable subgroups can be controlled with respect to taking certain HNN-extensions and amalgamated products.

\begin{lem}\label{HNNgen}
Suppose $G\in \X$, $\mathcal A_0\subseteq G$, and $S\leq G$ is suitable with respect to $\mathcal A_0$. Suppose also that $A$ and $B$ are cyclic subgroups of $G$. Then there exists $\mathcal A_0\subseteq \mathcal A\subseteq G$ such that $A\cup B\subseteq \mathcal A$ and $S$ is suitable with respect to $\mathcal A$.
\end{lem}
\begin{proof}
By Corollary \ref{suitsubc}, $S$ contains an infinite order element $y$ such that $\langle y\rangle\h (G, \mathcal A_0)$, and an element $g\in S\setminus\langle y\rangle$. By Lemma \ref{ee}, we can find a subset $\mathcal A_0\subseteq Y_0\subset G$ such that $\langle y\rangle\h (G, Y_0)$ and $A$ and $B$ are both elliptic with respect to the action of $G$ on $\Gamma(G, Y_0\sqcup\langle y\rangle)$. By Theorem \ref{Ahyp}, we can find $Y\supseteq Y_0$ such that $\langle y\rangle\h(G, Y)$, and the action of $G$ on $\Gamma(G, Y\sqcup\langle y\rangle)$ is acylindrical. Clearly $A$ and $B$ are still elliptic with respect to $\Gamma(G, Y\sqcup\langle y\rangle)$. By Lemma \ref{lox}, for some $n\in\N$, $gy^n$ is loxodromic with respect to $\Gamma(G, Y\sqcup\langle y\rangle)$. Thus, the action of $S$ on $\Gamma(G, Y\sqcup\langle y\rangle)$ is non-elementary by Theorem \ref{subah}. Letting $\mathcal A=(Y\cup A\cup B)\sqcup\langle y\rangle$, Lemma \ref{joinell} gives that $\Ga$ is hyperbolic, the action of $G$ on $\Ga$ is acylindrical, and the action of $S$ is still non-elementary, hence $S$ is suitable with respect to $\mathcal A$.
\end{proof}

\begin{prop}\label{HNNsuit}
Suppose $S$ is a suitable subgroup of a group $G\in\X$. Then for any isomorphic cyclic subgroups $A$ and $B$ of $G$, the corresponding HNN-extension $G\ast_{A^t=B}$ belongs to $\X$ and contains $S$ as a suitable subgroup. 
\end{prop} 

\begin{proof}
By Lemma \ref{HNNgen}, there exists $\mathcal A\subseteq G$ such that $S$ is suitable with respect to $\mathcal A$ and $A \cup B\subseteq \mathcal A$. Then Corollary \ref{suitsubc} gives that $S$ contains an element $h$ which is loxodromic with respect to $\Ga$ and which satisfies $E_G(h)=\langle h\rangle$.

Let $G_1$ denote the HNN-extension $G\ast_{A^t=B}$; we identify $G$ with its image in $G_1$. We will first show that $\Gamma(G_1, \mathcal A\cup\{t\})$ is a hyperbolic metric space. Since $\Ga$ is hyperbolic, by Theorem \ref{liniso} there exists a bounded presentation of $G$ of the form
\begin{equation}\label{Zpres}
G=\langle \mathcal A\;|\;\mathcal O\rangle
\end{equation}
such that for any word $W$ in $\mathcal A$ such that $W=_G 1$, the area of $W$ over the presentation (\ref{Zpres}) is at most $L\|W\|$ for some constant $L$. Then $G_1$ has the presentation
\begin{equation}\label{HNNpres}
G_1=\langle \mathcal A\cup\{t\}\;|\; \mathcal O\cup\{a^t=\phi(a)\;|\; a\in A\}\rangle
\end{equation}
where $\phi\colon A\to B$ is an isomorphism. Note that (\ref{HNNpres}) is still a bounded presentation, as we only added relations of length $4$ (we use here that $A\cup B\subseteq\mathcal A$). We will show that (\ref{HNNpres}) still satisfies a linear isoperimetric inequality, which is enough to show that $\Gamma(G_1, \mathcal A\cup\{t\})$ is a hyperbolic metric space by Theorem \ref{liniso}.

Let $W$ be a word in $\mathcal A\cup\{t\}$ such that $W=_{G_1} 1$. Let $\Delta$ be a minimal diagram over (\ref{HNNpres}). Note that every cell $\Pi$ of $\Delta$ which contains an edge labeled by $t$ forms a square with exactly two $t$-edges on $\partial\Pi$; we call such cells \emph{$t$-cells}. Hence these $t$-edges of $\Pi$ must either lie on $\partial\Delta$ or on the boundary of another $t$-cell. It follows that $\Pi$ belongs to a maximal, connected collection of $t$-cells, which we will call a $t$-band. Since $\Delta$ is minimal, it is well-known (and easy to prove) that every $t$-band of $\Delta$ starts and ends on $\partial\Delta$. Furthermore, since $A\cup B\subseteq \mathcal A$, minimality of $\Delta$ gives that each $t$-band consists of a single cell. Let $\Pi_1,...,\Pi_m$ denote the $t$-bands of $\Delta$. Then $\Delta\setminus\bigcup\Pi_i$ consists of $m+1$ connected components $\Delta_1,...,\Delta_{m+1}$ such that each $\Delta_i$ is a diagram over (\ref{Zpres}). Thus, for each $i$, $Area(\Delta_i)\leq L\l(\partial\Delta_i)$. Clearly $m\leq\l(\partial\Delta)$, and it is easy to see that
\[
\sum_{i=1}^{m+1}\l(\partial\Delta_i)=\l(\partial\Delta).
\]
It follows that
\[
Area(\Delta)=\sum_{i=1}^{m+1}Area(\Delta_i)+m\leq\sum_{i=1}^{m+1}L\l(\partial\Delta_i)+\l(\partial\Delta)\leq(L+1)\l(\partial\Delta).
\]
Thus, $Area(W)\leq (L+1)\|W\|$, and hence $\Gamma(G_1, \mathcal A\cup\{t\})$ is a hyperbolic metric space by Theorem \ref{liniso}.

Next, we will show that $h$ is loxodromic with respect to the action of $G_1$ on $\Gamma(G_1, \mathcal A\cup\{t\})$. Observe that any shortest word $W$ in $\mathcal A\cup\{t\}$ which represents an element of $G$ contains no $t$ letters. Indeed by Britton's Lemma if $W$ represents an element of $G$ and contains $t$ letters, then it has a subword of the form $t^{-1}at$ for some $a\in A$ or a subword of the form $tbt^{-1}$ for some $b\in B$. However, since $A\cup B\subseteq \mathcal A$, each of these subwords can be replaced with a single letter of $A\cup B$, contradicting the fact that $W$ is a shortest word. Since $h$ is loxodromic, if $W$ is the shortest word in $\mathcal A$ representing $h$ in $G$ then any path $p$ labeled by $W^n$ is quasi-geodesic in $\Gamma(G, \mathcal A)$. It follows that $p$ is still quasi-geodesic in $\Gamma(G_1, \mathcal A\cup\{t\})$ thus $h$ is loxodromic in $G_1$.

Finally, we will show that $h$ satisfies the WPD condition (\ref{WPDeq}) of Definition \ref{defn:WPD}; clearly it suffices to verify (\ref{WPDeq}) with $x=1$. Let $\e>0$, and choose $M$ such that if $x_1$ and $x_2$ satisfy $|x_i|_{\mathcal A} \leq \e$ for $i=1,2$, then $x_1Ax_2\cup x_1Bx_2\subseteq B_{\mathcal A}(M)$. Now choose $N_0$ such that for all $N\geq N_0$, $h^N\notin B_{\mathcal A}(M)$. Suppose $N\geq N_0$, and $g\in G_1$ such that $d_{\mathcal A\sqcup\{t\}}(1, g)\leq \e$ and $d_{\mathcal A\sqcup\{t\}}(h^N, gh^N)\leq \e$. Consider the quadrilateral $s_1p_1s_2(p_2)^{-1}$ in $\Gamma(G_1, \mathcal A\cup\{t\})$ where $\l(s_i)\leq\e$, $\Lab(s_1)=g$, and $\Lab(p_i)=h^N$. Without loss of generality, we assume each $s_i$ and each $p_i$ is a geodesic. As shown above, this means that no edges of $p_i$ are labeled by $t^{\pm 1}$. 

Suppose that $s_1$ contains an edge labeled by $t^{\pm 1}$. Filling this quadrilateral with a van Kampen diagram $\Delta$, for each edge of $s_1$ labeled by $t^{\pm 1}$, there exists a $t$-band connecting this edge to an edge of $s_2$. Let $e$ be the last $t$-edge of $s_1$, and let $r_1$ be the subpath of $s_1$ from $e_+$ to $(s_1)_+$. Similarly, let $r_2$ be the subpath of $s_2$ from $(s_2)_-$ to $f_-$, where $f$ is the $t$-edge of $s_2$ connected to $e$ by a $t$-band. Let $q$ be the path from $e_+$ to $f_-$ given by the $t$-band joining these edges. This means that $\Lab(q)$ is an element of $A$ or $B$; for concreteness we assume it is equal to an element $a\in A$. Note that $r_1$ and $r_2$ cannot contain edges labeled by $t^{\pm 1}$, otherwise $e$ would not be the last $t$-edge of $s_1$. Let $x_i$ be the element of $G$ given by $\Lab(r_i)$ for $i=1,2$. Then $p_1r_2q^{-1}r_1$ forms a cycle in $\Gamma(G_1, \mathcal A\cup\{t\})$, and moreover no edge of this cycle is labeled by $t^{\pm 1}$. Thus,
\[
h^N=x_1^{-1}ax_2^{-1}
\]
where this equality holds in $G$. However, this violates our choice of $N$. Therefore, $s_1$ must not contain any $t$-letters, and hence $g\in G$. Thus,
\[
\{g\in G_1\;|\;d_{\mathcal A\sqcup\{t\}}(1, g)<\e, d_{\mathcal A\sqcup\{t\}}(h^N, gh^N)<\e\}\subseteq \{g\in G\;|\;d_{\mathcal A}(1, g)<\e, d_{\mathcal A}(h^N, gh^N)<\e\},
\]
and this last set is finite (for sufficiently large $N$) because $h$ satisfies WPD with respect to the action of $G$ on $\Gamma(G, \mathcal A)$. Thus, $h$ is a loxodromic, WPD element with respect to the action of $G_1$ on $\Gamma(G_1, \mathcal A\cup\{t\})$, hence $G_1\in\X$ by Theorem \ref{thmAH}.
 
Since $h$ is loxodromic with respect to the action of $G$ on $\Gamma(G, \mathcal A)$, it is not conjugate with any elliptic element; in particular it is not conjugate with any element of $A$ or $B$. It follows from Lemma \ref{E(h)} (and conjugacy in HNN-extensions, for example \cite[Lemma 2.14]{HO1}) that $E_{G_1}(h)=E_G(h)=\langle h\rangle$. Therefore $S$ is a suitable subgroup of $G_1$ by Lemma \ref{suitsubl}. 
\end{proof}

We now prove a similar result for amalgamated products using a standard retraction trick. The following lemma is a simplification of \cite[Lemma 6.21]{DGO}; recall that a subgroup $H\leq G$ is called a \emph{retract} if there exists a homomorphism $r\colon G\to H$ such that $r^2=r$.
\begin{lem}\label{heret}
Suppose $G$ is a group, $R$ a subgroup which is a retract of $G$, and $H\leq R$ such that $H\h G$. Then $H\h R$.
\end{lem}

The following is well-known; it can be easily derived from the proof of \cite[ Chapt. IV, Theorem 2.6]{LS}.
 \begin{thm}\label{rtrick}
Let $P=A\ast_{K=J}B$, and let $G=(A\ast B)\ast_{K^t=J}$; that is, $G$ is an HNN extension of the free product $A\ast B$. Then $P$ is naturally isomorphic to the retract $\langle A^t, B\rangle\leq G$.

\end{thm}

\begin{prop}\label{amalsuit}
Suppose $A\in\X$ and $S$ is a suitable subgroup of $A$. Let $P=A\ast_{K=\phi(K)}B$, where $K$ is cyclic. Then $P\in\X$ and $S$ is a suitable subgroup of $P$.
\end{prop}

\begin{proof}
Clearly, $A\h A\ast B$, so by Lemma \ref{suithe}, if $S$ is suitable in $A$, then $S$ is suitable in $A\ast B$. By the previous lemma, $S$ is suitable in the HNN extension $G=(A\ast B)\ast_{K^t=\phi(K)}$. By Theorem \ref{rtrick} $P$ is isomorphic to $\langle A^t, B\rangle\leq G$ via an isomorphism which sends $A$ to $A^t$ and $B$ to $B$. Furthermore, $\langle A^t, B\rangle$ is a retract of $G$. Thus if $h\in S\leq A$ satisfies $\langle h\rangle\h G$, then $\langle h^t\rangle\h G$ by Lemma \ref{conjhe} and $\langle h^t\rangle\h\langle A^t, B\rangle$ by Lemma \ref{heret}. Thus $S^t$ is a suitable subgroup of $\langle A^t, B\rangle$ by Lemma \ref{suitsubl}, and passing to $P$ through the isomorphism gives the desired result.

\end{proof}



\section{Main theorem and applications}\label{sect:7}
 
\begin{thm}\label{scthm}
Suppose $G\in\X$ and $S$ is suitable with respect to $\mathcal A$. Then for any $\{t_1,...,t_m\}\subset G$ and $N\in\N$, there exists a group $\overline{G}$ and a surjective homomorphism $\gamma\colon G\to \overline{G}$ which satisfy
\begin{enumerate}
\item[(a)] $\overline{G}\in\X$.
\item[(b)] $\gamma|_{B_{\mathcal A}(N)}$ is injective.
\item[(c)] $\gamma(t_i)\in\gamma(S)$ for $i=1,...,m$.
\item[(d)] $\gamma(S)$ is suitable with respect to $\mathcal A^\prime$, where $\gamma(\mathcal A)\subseteq\mathcal A^\prime$.
\item[(e)] Every element of $\overline{G}$ of order $n$ is the image of an element of $G$ of order $n$.
\end{enumerate} 
\end{thm}

\begin{proof}
Clearly it suffices to prove the theorem with $m=1$, and the general statement follows by induction. Since $S$ is suitable with respect to $\mathcal A$, by Corollary \ref{suitsubc} $S$ contains infinite order elements $h_1$ and $h_2$ such that $\{\langle h_1\rangle, \langle h_2\rangle\}\h (G,\mathcal A)$. Let $t=t_1$ and $\mathcal A_1=(\mathcal A\cup\{t^{\pm 1}\})\sqcup\langle h_1\rangle\sqcup\langle h_2\rangle$, and fix $\e$, $\mu$, and $\rho$ satisfying the conditions of Lemma \ref{scquot} and Lemma \ref{tor} for $\lambda=\frac{1}{4}$, $c=1$, and $N$. Choose $n$ such that $\frac{M}{2n}\leq\mu$ and $2n\geq \rho$, where $M=M(\e, 9)$ is the constant given by Lemma \ref{conscomp}. Now if $m_1,...,m_n$ and $l_1,...,l_n$ are sufficiently large, distinct positive integers, then the word
\[ 
W\equiv t^{-1}h_1^{m_1}h_2^{l_1}...h_1^{m_n}h_2^{l_n}
\]
 will satisfy all the assumptions of Proposition \ref{scwords} (here $W$ is considered as a word in $\mathcal A_1$). Thus, the set $\mathcal R$ of all cyclic shifts of $W^{\pm 1}$ satisfies the $C^\prime(\e,\frac{M}{2n},\frac{1}{4},1,2n)$-condition by Proposition \ref{scwords}. Let 
 \[
 \overline{G}=G/\ll\mathcal R\rr
 \]
 and let $\gamma\colon G\to \overline{G}$ be the natural homomorphism. Lemma \ref{scquot} gives that $\gamma$ is injective on $B_{\mathcal A_1}(N)$, and hence it is also injective on $B_{\mathcal A}(N)$. Lemma \ref{scquot} also gives that $\{\gamma(\langle h_1\rangle), \gamma(\langle h_2\rangle)\}\h\overline{G}$, thus $\overline{G}\in\X$ by Theorem \ref{thmAH}. Lemma \ref{tor} gives that  every element of $\overline{G}$ of order $n$ is the image of an element of $G$ of order $n$. Furthermore, since $t^{-1}h_1^{m_1}h_2^{l_2}...h_1^{m_n}h_2^{l_n}\in \mathcal R$, we have that $\gamma(t)=\gamma(h_1^{m_1}h_2^{l_2}...h_1^{m_n}h_2^{l_n})\in\gamma(S)$. Finally, Lemma \ref{scquot} gives that $\gamma(\langle h_1\rangle)=\langle \gamma(h_1)\rangle\h(\overline{G}, \gamma(\mathcal A))$. Since $\gamma(h_2)\notin\langle \gamma(h_1)\rangle$, $\gamma(S)$ is suitable with respect to $\mathcal A^\prime$ for some $\mathcal A^\prime\supseteq\gamma(\mathcal A)$ by Lemma \ref{suitsubl}.
 
\end{proof}

\begin{rem}\label{remsamecond}
Since the proof uses the same small cancellation conditions as \cite{Ols} and \cite{Osi10}, it follows from these papers that if $G$ is non-virtually-cyclic and hyperbolic, then $\overline{G}$ can be chosen non-virtually-cyclic and hyperbolic, and if $G$ is hyperbolic relative to $\Hl$, then $\overline{G}$ can be chosen hyperbolic relative to $\{\gamma(H_{\lambda})\}_{\lambda\in\Lambda}$.
\end{rem}

Note that we can always choose $N$ such that $B_{\mathcal A}(N)$ contains any given finite subset of $G$. Also, if $G$ is finitely generated, we can choose $\{t_1,...,t_m\}$ to be a generating set of $G$, and we get that $\gamma|_S$ is surjective. If $G$ is countable but not finitely generated, we can apply this theorem inductively to to get a similar result, although the limit group may not be acylindrically hyperbolic.

\begin{cor}\label{ontosuit}
Suppose $G\in\X$ is countable and $S$ is suitable with respect to $\mathcal A$. Then for any $N\in\N$, there exists a non-virtually-cyclic group $Q$ and a surjective homomorphism $\eta\colon G\to Q$ such that 
\begin{enumerate}
\item
$\eta|_S$ is surjective.
\item
$\eta|_{B_{\mathcal A}(N)}$ is injective.
\end{enumerate}

\end{cor}

\begin{proof}

By Lemma \ref{HNNgen}, without loss of generality, we can assume that $\mathcal A$ contains infinite cyclic subgroups $\langle f\rangle$ and $\langle g\rangle$ such that $\langle f\rangle\cap\langle g\rangle=\{1\}$.

Let $G=\{1=g_0, g_1,...\}$. Let $G_0=G$, and define a sequence of quotient groups
\[
...\twoheadrightarrow G_i\twoheadrightarrow G_{i+1}\twoheadrightarrow...
\]
where the induced map $\eta_i \colon G\twoheadrightarrow G_i$ satisfies
\begin{enumerate}

\item $\eta_i(S)$ is suitable with respect to $\mathcal A_i$, where $\eta_i(\mathcal A)\subseteq \mathcal A_i$.

\item $\eta_i(g_i)\in\eta_i(S)$.

\item $\eta_i|_{B_{\mathcal A}(N)}$ is injective.

\end{enumerate}
Given $G_i$, we apply Theorem \ref{scthm} to $G_i$ with $t=\eta_i(g_{i+1})$ and suitable subgroup $\eta_i(S)$, and let $G_{i+1}=\overline{G_i}$. Theorem \ref{scthm} gives that the map $\gamma\colon G_i\to G_{i+1}$ will be injective on $B_{\mathcal A_i}(N)$ which contains $B_{\eta_i(\mathcal A)}(N)$, and further for some $\mathcal A_{i+1}\supset \gamma(\mathcal A_i)$, $\gamma(\eta_i(S))$ is suitable with respect to $\mathcal A_{i+1}$. Hence the induced quotient map $\eta_{i+1}=\gamma\circ\eta_i$ will satisfy all of the above conditions. Let $Q$ be the direct limit of this sequence, that is, $Q=G_0/\bigcup_{i=1}^{\infty}\ker \eta_i$. Let $\eta\colon G\twoheadrightarrow Q$ be the induced epimorphism. Then for each $g_i\in G$, $\eta_i(g_i)\in\eta_i(S)$, thus $\eta(g_i)\in \eta(S)$. It follows that $\eta|_S$ is surjective. Finally, $\eta|_{B_{\mathcal A}(N)}$ is injective, since each $\eta_i$ is  injective on $B_{\mathcal A}(N)$. Thus $\eta$ is injective on $\langle f\rangle\cup\langle g\rangle\subseteq B_{\mathcal A}(N)$, so $Q$ is not virtually cyclic.
\end{proof}

\begin{cor}\label{commonq}
Let $G_1,G_2\in\X$ with $G_1$ finitely generated, $G_2$ countable. Then there exists a non-virtually cyclic group $Q$ and surjective homomorphisms $\alpha_i\colon G_i\to Q$ for $i=1,2$. In addition, if $G_2$ is finitely generated, then we can choose $Q\in\X_0$, and if $K(G_i)=\{1\}$, then for any finite subset $\mathcal F_i\subset G_i$, we can choose $\alpha_i$ to be injective on $\mathcal F_i$.
\end{cor}

\begin{proof}
Since each $G_i$ can be replaced with $G_i/K(G_i)$ by Lemma \ref{k(g)}, it suffices to assume $K(G_i)=\{1\}$ for $i=1, 2$. Let $\mathcal F_i$ be any finite subset of $G_i$. Let $F=G_1\ast G_2$, and let $\iota_i\colon G_i\to F$ be the natural inclusion. We will identify $G_1$ and $G_2$ with their images in $F$. By Corollary \ref{suitsubc}, there exist infinite order elements $f_1$, $f_2\in G_1$ such that $\{\langle f_1\rangle, \langle f_2\rangle\}\h G_1$ and infinite order elements $h_1$, $h_2\in G_2$ such that $\{\langle h_1\rangle, \langle h_2\rangle\}\h G_2$. Since $\{G_1, G_2\}\h F$, Lemma \ref{hehe} gives that $\{\langle f_1\rangle, \langle f_2\rangle, \langle h_1\rangle, \langle h_2\rangle\}\h F$. Thus, $S=\langle h_1, h_2\rangle$ is suitable in $F$ by Lemma \ref{suitsubl}.

Let $t_1,...,t_m$ be a finite generating set of $G_1$. By Theorem \ref{scthm}, there exists a group $F^\prime$ and a surjective homomorphism $\gamma\colon F\to F^\prime$ such that $\gamma|_{\mathcal F_1\cup\mathcal F_2}$ is injective and $\gamma(t_i)\in\gamma(S)$ for each $1\leq i\leq m$. In particular, $\gamma(G_1)\subseteq \gamma(S)\subseteq \gamma(G_2)$. 

It is clear from the proof of Theorem \ref{scthm} that $F^\prime$ can be formed by setting each $t_i$ equal to a small cancellation word in $\{h_1, h_2\}$. Since $\{\langle f_1\rangle, \langle f_2\rangle, \langle h_1\rangle, \langle h_2\rangle\}\h F$, it follows from Lemma \ref{scquot} that we can choose $F^\prime$ such that $\{\langle \gamma(f_1)\rangle, \langle \gamma(f_2)\rangle, \langle \gamma(h_1)\rangle, \langle \gamma(h_2)\rangle\}\h F^\prime$. Since $f_1, f_2\in G_1$ and $\{\langle \gamma(f_1)\rangle, \langle \gamma(f_2)\rangle\}\h F^\prime$, Lemma \ref{suitsubl} gives that $\gamma(G_1)$ is suitable in $F^\prime$.

Now applying Corollary \ref{ontosuit} to $F^\prime$ with $\gamma(G_1)$ as a suitable subgroup gives a non-virtually cyclic group $Q$ and a surjective homomorphism $\eta\colon F^\prime\to Q$, such that $\eta|_{\gamma(G_1)}$ is surjective and $\eta|_{\gamma(\mathcal F_1)\cup\gamma(\mathcal F_2)}$ is injective. Now since $\gamma(G_1)\subseteq \gamma(G_2)$ and $\eta|_{\gamma(G_1)}$ is surjective, it follows that each of the compositions

\[
G_i\stackrel{\iota_i}\hookrightarrow F\stackrel{\gamma}\twoheadrightarrow F^\prime\stackrel{\eta}\twoheadrightarrow Q
\]
is surjective. Furthermore, each composition $\eta\circ\gamma\circ\iota_i$ is injective on $\mathcal F_i$. Now if $G_2$ is finitely generated, then $F^\prime$ is finitely generated and we can apply Theorem \ref{scthm} to $F^\prime$ with a generating set of $F^\prime$ as a finite set of elements to get $Q$ such that the image of $G_1$ maps onto $Q$. Then we will also get that $Q\in\X$ and the image of $G_1$ is a suitable subgroup, thus $Q\in\X_0$.
\end{proof}

\paragraph{Frattini subgroups.}

Recall that  $Fratt(G)=\{g\in G\;|\; g \text{ is a non-generator of } G\}$, where an element $g\in G$ is called a \emph{non-generator} if for all $X\subseteq G$ such that $\langle X\rangle=G$, we have $\langle X\setminus \{g\}\rangle=G$. Conversely, if $X$ is a generating set of $G$ such that $\langle X\setminus \{g\}\rangle\neq G$, then we say that $g$ is an \emph{essential member} of the generating set $X$.

\begin{lem}\label{fratquot}
Let $\phi\colon G\to G^\prime$ be a homomorphism. If $\phi(g)\notin Fratt(\phi(G))$, then $g\notin Fratt(G)$.
\end{lem}

\begin{proof}
Suppose $\phi(g)$ is an essential member of a generating set $Y$ of $\phi(G)$. Choose $X\subseteq G$ such that $g\in X$, $\phi(X)=Y$, and $\phi|_X$ is injective. Then $g$ is an essential member of the generating set $X\cup\ker(\phi)$ of $G$.
\end{proof}

\begin{thm}\label{thm:Fratt}
Let $G\in\X$ be countable. Then $Fratt(G)\leq K(G)$; in particular, the Frattini subgroup is finite.
\end{thm}

\begin{proof}
First, we assume that $K(G)=\{1\}$ and let $g\in G\setminus\{1\}$. Since $K(G)=\{1\}$, Corollary \ref{suitsubc} gives that $G$ contains infinite order elements $h_1$ and $h_2$ such that $\langle h_1\rangle\cap \langle h_2\rangle=\{1\}$ and $\{\langle h_1\rangle, \langle h_2\rangle\}\h G$. In particular, this means that $G$ contains some infinite order element $h$ such that $\langle h\rangle\h G$ and $g\notin \langle h\rangle$. Let $S=\langle g, h\rangle$. By Lemma \ref{suitsubl}, $S$ is a suitable subgroup of $G$. Now we can apply Corollary \ref{ontosuit} to find a non-virtually-cyclic group $Q$  and a homomorphism $\eta\colon G\to Q$ such that $\eta|_S$ is surjective, thus $Q$ is generated by $X=\{\eta(g), \eta(h)\}$. Now $\eta(g)$ is an essential member of the generating set $X$ since $Q$ is not cyclic, so $\eta(g)\notin Fratt(\eta(G))$. Therefore $g\notin Fratt(G)$ by Lemma \ref{fratquot}.

Now consider any countable $G\in\X$ and let $g\in G\setminus K(G)$. By Lemma \ref{k(g)}, $G/K(G)\in\X_0$, so as above the image of $g$ does not belong to $Fratt(G/K(G))$. Hence by Lemma \ref{fratquot} $g\notin Fratt(G)$.

\end{proof}

\paragraph{Topology of marked group presentations.}

Recall that $\mathcal G_k$ denotes the set of \emph{marked $k$-generated groups}, that is 
\[
\mathcal G_k=\{(G, X)\;|\; \text{$X\subseteq G$ is an ordered set of $k$ elements and } \langle X\rangle= G\}.
\]

Each element of $\mathcal G_k$ can be naturally associated to a normal subgroup $N$ of the free group on $k$ generators by the formula
\[
G=F(X)/N.
\] 

Given two normal subgroups $N$, $M$ of the free group $F_k$, we can define a distance
\[
d(N, M)=\begin{cases} \max\left\{\frac{1}{\|W\|}\;|\; W\in N \Delta M\right\}&\text{ if } M\neq N \\
 0 &\text{ if } M=N

\end{cases}
\]

This defines a metric (and hence a topology) on $\mathcal G_k$. It is not hard to see that this topology is equivalent to saying that a sequence $(G_n, X_n)\rightarrow (G, X)$ in $\mathcal G_k$ if and only if there are functions $f_n\colon\Gamma(G, X_n)\to\Gamma(G, X)$ which are label-preserving isometries between increasingly large neighborhoods of the identity.

Recall that given a class of groups $\mathcal X$, $[\mathcal X]_k=\{(G, X)\in\mathcal G_k\;|\; G\in\mathcal X\}$. In case $\mathcal X$ consists of a single group $G$, we denote $[\mathcal X]_k$ by $[G]_k$. Also, $[\mathcal X]=\bigcup_{i=1}^\infty[\mathcal X]_k$ and $\overline{[\mathcal X]}=\bigcup_{i=1}^\infty\overline{[\mathcal X]}_k$, where $\overline{[\mathcal X]}_k$ denotes the closure of $[\mathcal X]_k$ in $\mathcal G_k$.

\begin{thm}\label{dense}
Let $\mathcal C$ be a countable subset of $[\X_0]$. Then there exists a finitely generated group $D$ such that $\mathcal C\subset \overline{[D]}$.
\end{thm}

\begin{proof}
We begin by enumerating the set $\mathcal C\times\N=\{((G_1, X_1), n_1),...\}$

Let $Q_1=G_1$, and suppose we have defined groups $Q_1,..., Q_m$ and for each $Q_k$, we have surjective homomorphisms $\alpha_{(k,k)}\colon G_k\twoheadrightarrow Q_k$ and $\beta_{(k-1,k)}\colon Q_{k-1}\twoheadrightarrow Q_k$.

For $i\leq j$, let $\beta_{(i,j)}$ be the natural quotient map from $Q_i$ to $Q_j$, and let $\alpha_{(i,j)}=\beta_{(i,j)}\circ\alpha_{(i,i)}$. Suppose that for each $1\leq k\leq m$, $Q_k$ satisfies
\begin{enumerate}

\item $Q_k\in \X_0$.

\item for each $1\leq i\leq k$, $\alpha_{(i, k)}|_{B_{X_{i}}(n_i)}$ is injective.

\end{enumerate}

Let $\mathcal F=\cup_{i=1}^{m-1}\alpha_{(i, m)}(B_{X_{i}}(n_i))\subset Q_m$; note that $\mathcal F$ is finite since it is a finite union of finite sets. Now, by Corollary \ref{commonq} there exists a group $Q_{m+1}$ and surjective homomorphisms $\beta_{(m, m+1)}\colon Q_m\twoheadrightarrow Q_{m+1}$ and $\alpha_{(m+1,m+1)}\colon G_{m+1}\twoheadrightarrow Q_{m+1}$, such that $Q_{m+1}\in \X_0$, $\beta_{(m, m+1)}$ is injective on $\mathcal F$ and $\alpha_{(m+1,m+1)}$ is injective on $B_{X_{m+1}}(n_{m+1})$. Thus the above conditions are satisfied for $Q_{m+1}$. 

\begin{center}
$\begin{array}[c]{ccccccccccc}
G_1 && G_2 &&&& G_m && \\
\phantom{\scriptstyle{\alpha_{(1,1)}}}\twoheaddownarrow\scriptstyle{\alpha_{(1,1)}} && \phantom{\scriptstyle{\alpha_{(1,1)}}}\twoheaddownarrow\scriptstyle{\alpha_{(2,2)}} &&&& \phantom{\scriptstyle{\alpha_{(1,1)}}}\twoheaddownarrow\scriptstyle{\alpha_{(m,m)}} &&\\
Q_1 & \stackrel{\beta_{(1,2)}}{\twoheadrightarrow} & Q_2 & \stackrel{\beta_{(2,3)}}{\twoheadrightarrow} & ... & \stackrel{\beta_{(m-1,m)}}{\twoheadrightarrow} &  Q_m & \twoheadrightarrow & ... & \twoheadrightarrow D
\end{array}$ 
\end{center}

Now define $D$ to be the direct limit of the sequence $Q_1,...$. That is, $D=Q_1/\bigcup_{n=1}^\infty\ker\beta_{1,n}$. Let $\eta_i\colon G_i\twoheadrightarrow D$ denote the composition of $\alpha_{(i,i)}$ and the natural quotient map from $Q_i$ to $D$. Let $Y_i=\eta_i(X_{i})$. We will show that $\eta_i$ bijectively maps $B_{X_{i}}(n_i)\subset \Gamma(G_i, X_{i})$ to $B_{Y_i}(n_i)\subset \Gamma(D, Y_i)$. Clearly $\eta_i$ is surjective. now suppose $g, h\in B_{X_{i}}(n_i)$, $g\neq h$ and $\eta_i(g)=\eta_i(h)$. This means that $\alpha_{(i,i)}(gh^{-1})\in \bigcup_{n=i}^\infty\ker\beta_{i,n}$, thus there must exist some $k\geq i$ such that $\beta_{(i,k)}(\alpha_{(i,i)}(g))=\beta_{(i,k)}\alpha_{(i,i)}(h)$. But this means that $\alpha_{(i, k)}(g)=\alpha_{(i, k)}(h)$, which contradicts one of our inductive assumptions. Thus, $\eta_i$ bijectively maps $B_{X_{i}}(n_i)$ to $B_{Y_i}(n_i)$.

Now let $(G, X)\in\mathcal C$, and let $((G_{i_j}, X_{i_j}), n_{i_j})$ be the subsequence corresponding to $(G, X)$. Note that each $X_{i_j}=X$, so $\eta_{i_j}$ bijectively maps $B_{X}(n_{i_j})\subset\Gamma(G, X)$ to $B_{Y_{i_j}}(n_{i_j})\subset\Gamma(D, Y_{i_j})$.

Therefore,

\[
\lim_{j\rightarrow \infty}(D, Y_{i_j})=(G, X).
\]

\end{proof}

\paragraph{Exotic quotients.}
Recall that a group $G$ is called \emph{verbally complete} if for any $k\geq 1$, any $g\in G$, and any freely reduced word $W(x_1,..., x_k)$ there exist $g_1,...,g_k\in G$ such that $W(g_1,..., g_k)=g$ in the group $G$.

\begin{thm}\label{thm:vcquot}
Let $G\in\X$ be countable. Then $G$ has a non-trivial finitely generated quotient $V$ such that $V$ is verbally complete.
\end{thm}

\begin{proof}
By Lemma \ref{k(g)} we can assume $K(G)=\{1\}$. By Corollary \ref{suitsubc}, $G$ contains an infinite order element $h$ such that $\langle h\rangle \h G$. Let $h^\prime\in G\setminus \langle h\rangle$, and let $S=\langle h, h^\prime\rangle$. Then $S$ is a suitable subgroup by Lemma \ref{suitsubl}. Enumerate all pairs $\{(g_1, v_1),...\}$ where $g_i\in G$ and $v_i=v_i(x_1,...)$ is a non-trivial freely reduced word in $F(x_1,...)$. Let $G(0)=G$, and suppose we have constructed $G(n)$ and a surjective homomorphism $\alpha_n\colon G\twoheadrightarrow G(n)$ satisfying
\begin{enumerate}
\item $G(n)\in\X$.

\item $\alpha_n(S)$ is a suitable subgroup  of $G(n)$.

\item The equation $g_i=v_i(x_1,...)$ has a solution in $G(n)$ for each $1\leq i\leq n$. 

\item $\alpha_n(g_i)\in\alpha_n(S)$ for each $1\leq i\leq n$.

\end{enumerate}

Given $G(n)$, choose $m$ such that $v_{n+1}$ is a word in $x_1,..., x_m$, and let $J=F(x_1,..., x_m)$ if $g_{n+1}$ has infinite order, and $J=\langle x_1,...x_m\;|\; v_{n+1}^k=1\rangle$ if $g_{n+1}$ has order $k$. In the case where $g_{n+1}$ has order $k$, it is well-known that the order of $v_{n+1}$ in $J$ is $k$ (see \cite[ Chapt. IV, Theorem 5.2]{LS}). Thus the amalgamated product $G(n+\frac{1}{2})=G(n)\ast_{g_{n+1}=v_{n+1}} J$ is well-defined in either case. By Lemma \ref{amalsuit}, $\alpha_n(S)$ is a suitable subgroup of $G(n+\frac{1}{2})$, so we can apply Theorem \ref{scthm} to get a group $G(n+1)\in\X$ and a surjective homomorphism $\gamma\colon G(n+\frac12)\to G(n+1)$ such that $\gamma(\alpha_n(S))$ is suitable, and $\{\gamma(x_1),...,\gamma(x_m),\gamma( g_{n+1})\}\subset \gamma(\alpha_n(S))$. Since $G(n+\frac12)$ is generated by $\{G(n), x_1,...,x_m\}$ and $\gamma(x_i)\in\gamma(G(n))$ for each $1\leq i\leq m$, it follows that the restriction of $\gamma$ to $G(n)$ is surjective. Thus there is a natural quotient map $\alpha_{n+1}\colon G\twoheadrightarrow G(n+1)$. Since $g_{n+1}=v_{n+1}(x_1,...)$ has a solution in $G(n+\frac12)$, it also has a solution in $G(n+1)$; the other inductive assumptions follow from Theorem \ref{scthm}. Let $V$ be the direct limit of the sequence $G(0),...$, and let $\alpha\colon G\to V$ be the natural quotient map. For each $g\in G$, there exists $n$ such that $\alpha_n(g)\in\alpha_n(S)$; thus, the restriction of $\alpha$ to $S$ is surjective, so $V$ is two-generated. Also for any non-trivial, freely reduced word $v(x_1,...)$ in $F(x_1,...)$ and any $g\in G$, there exists $n$ such that $g=v(x_1,...)$ has a solution in $G(n)$, and hence this equation has a solution in $V$. Thus $V$ is verbally complete. Finally, suppose $V$ is trivial. Then for some $n$, $\alpha_n(h)=\alpha_n(h^\prime)=1$; but since $S=\langle h, h^\prime\rangle$, this means that $\alpha_n(S)=\{1\}$, contradicting the fact that $\alpha_n(S)$ is a suitable subgroup  of $G(n)$. Hence $V$ is non-trivial.
\end{proof}

\begin{thm}\label{conjquot}
Let $G\in\X$ be countable. Then $G$ has an infinite, finitely generated quotient $C$ such that any two elements of $C$ are conjugate if and only if they have the same order and $\pi(C)=\pi(G)$. In particular, if $G$ is torsion free, then $C$ has two conjugacy classes.
\end{thm}

\begin{proof}
We first assume $K(G)=\{1\}$.  As in the previous Theorem, Corollary \ref{suitsubc} and Lemma \ref{suitsubl} imply that $G$ contains a two-generated suitable subgroup $S$. Let $\mathcal A\subseteq G$ be a generating set of $G$ such that $S$ is suitable with respect to $\mathcal A$. By Lemma \ref{nolongell}, for all $k\in\pi(G)\setminus\{\infty\}$, there exists $f_k\in G$ such that $f_k$ has order $k$ and $|f_k|_{\mathcal A}\leq 8\delta$, where $\delta$ is the hyperbolicity constant of $\Ga$. By Lemma \ref{HNNgen}, we can assume that $\mathcal A$ contains an infinite cyclic subgroup $\langle f_{\infty}\rangle$.  Let $\mathcal O=\{f_k\;|\;k\in\pi(G)\}$. Now enumerate $G$ as $\{1=g_0, g_1,...\}$. Let $G(0)=G$, and suppose we have constructed $G(n)$ and a surjective homomorphism $\alpha_n\colon G\twoheadrightarrow G(n)$ satisfying:

\begin{enumerate}
\item $G(n)\in\X$.

\item $\alpha_n(S)$ is a suitable subgroup  of $G(n)$.

\item $\pi(G(n))=\pi(G)$, and for all $k\in\pi(G)$, $\alpha_n(f_k)$ has order $k$.

\item For each $1\leq i\leq n$, $\alpha_n(g_i)$ is conjugate to an element of $\alpha_n(\mathcal O)$ and $\alpha_n(g_i)\in\alpha_n(S)$.

\end{enumerate}
We construct $G(n+1)$ in two steps. First, if $\alpha_n(g_{n+1})$ is conjugate to an element of $\alpha_n(\mathcal O)$, set $G(n+\frac12)=G(n)$. Otherwise, choose $k\in\pi(G)$ such that $\alpha_n(g_{n+1})$ has order $k$, and let $G(n+\frac{1}{2})$ be the HNN-extension $G\ast_{\alpha_n(g_{n+1})^t=\alpha_n(f_k)}$. We identify $G(n)$ with its image inside $G(n+\frac12)$, and by Lemma \ref{HNNsuit}, $\alpha_n(S)$ is a suitable subgroup of $G(n+\frac{1}{2})$. 

Applying Theorem \ref{scthm} to $G(n+\frac{1}{2})$ with $\alpha_n(S)$ as a suitable subgroup and $\{t, \alpha_n(g_{n+1})\}$ (or just $\{\alpha_n(g_{n+1})\}$ if $G(n+\frac12)=G(n)$) as a finite set of elements and $N=8\delta$ produces a group $G(n+1)\in \X$ and a surjective homomorphism $\gamma\colon G(n+\frac12)\twoheadrightarrow G(n+1)$, such that $\gamma(t)$, $\gamma(\alpha_n(g_{n+1}))\in\gamma(\alpha_n(S))$ and $\gamma(\alpha_n(S))$ is a suitable subgroup of $G(n+1)$.  Since $G(n+\frac12)$ is generated by $G(n)$ and $t$ and $\gamma(t)\in\gamma(G(n))$, it follows that the restriction of $\gamma$ to $G(n)$ is surjective. Let $\alpha_{n+1}=\gamma\circ\alpha_n$.

Note that for each $f_k\in\mathcal O$ and each $1\leq j\leq k$, $f_k^j$ is conjugate to an element inside $B_{\mathcal A}(8\delta)$ and since $\alpha_{n+1}$ is injective on $B_{\mathcal A}(8\delta)$, the order of $\alpha_{n+1}(f_k)$ is $k$. Applying this along with the last condition of Theorem \ref{scthm} gives that $\pi(G(n+1))=\pi(G)$. Thus $G(n+1)$ will satisfy the inductive assumptions. 

Let $C$ be the direct limit of the sequence $G(1),...$, and let $\alpha\colon G\twoheadrightarrow C$ be the natural quotient map. First note that for each $g_i\in G$, $\alpha_i(g_i)\in\alpha_i(S)$, thus $\alpha(g_i)\in\alpha(S)$. Therefore the restriction of $\alpha$ to $S$ is surjective; in particular, $C$ is two-generated. By condition (3), $\alpha(f_k)$ has order $k$ and $\pi(C)=\pi(G)$. 

Suppose $x$ and $y$ are elements of order $k$ in $C$. Let $g_i$ be a preimage of $x$ and $g_j$ a preimage of $y$ in $G$. Then in $G(i)$, $\alpha_i(g_i)$ is conjugate to $\alpha_i(f_{k^\prime})$ for some $f_{k^\prime}\in\mathcal O$, hence $x$ is conjugate to $\alpha(f_{k^\prime})$. Since $f_{k^\prime}$ and $\alpha(f_{k^\prime})$ have the same order, we get that $k=k^\prime$. Thus $x$ is conjugate to $\alpha(f_k)$, and by the same argument so is $y$. Thus $x$ and $y$ are conjugate.

Finally, in order to remove the assumption that $K(G)=\{1\}$, we replace $G$ with 
\[
G^\prime=G/K(G)\ast\left(\ast_{n\in\pi(K(G))}\mathbb Z/n\mathbb Z \right).
\]
That is, $G^\prime$ is the free product of $G/K(G)$ and cyclic groups which each correspond to the order of an element of $K(G)$. Note that $K(G^\prime)=\{1\}$ and $\pi(G^\prime)=\pi(G)$. Lemma \ref{suithe} gives that any suitable subgroup of $G/K(G)$ is still suitable in $G^\prime$. Hence the two-generated suitable subgroup $S$ can be chosen as a subgroup of $G/K(G)$. Then applying the above construction yields the desired group $C$ and quotient map $\alpha$. Since the restriction of $\alpha$ to $S$ is surjective, $C$ is a quotient of $G/K(G)$ and hence a quotient of $G$.

\end{proof}

Given a subset $\mathcal S\subseteq \mathcal G_k$, a group property is said to be \emph{generic in S} if this property holds for all groups belonging to some dense $G_\delta$ subset of $\mathcal S$. 
Let $\X_{tf}$ denote the class of torsion free acylindrically hyperbolic groups. A version of the following corollary was suggested for relatively hyperbolic groups in the final paragraph of \cite{OOK}, and our proof is essentially the same as the proof sketched there.

\begin{cor}\label{generic}
A generic group in $\overline{[\X_{tf}]}_k$ has two conjugacy classes.
\end{cor}
\begin{proof}
In \cite{OOK}, it is shown that groups which have two conjugacy classes form a $G_\delta$ subset of $\mathcal G_k$. Hence we only need to show that such groups are dense in $\overline{[\X_{tf}]}_k$.

Let $G\in\X_{tf}$ be generated by $X=\{x_1,...,x_k\}$. Fix $N\in\N$ and let $G=G(1), G(2),....$ be the sequence constructed in the proof of Theorem \ref{conjquot}. By Theorem \ref{scthm}, we can ensure that the quotient map $G(i)\twoheadrightarrow G(i+1)$ is injective on $B_X(N+i)$, where the set $X$ is identified with its image in each quotient. It follows that $B_X(N+i)$ in $G(i)$ maps bijectively onto $B_X(N+i)$ in $C$, thus 
\[
\lim_{i\rightarrow\infty}(G_i, X)=(C, X)
\]
where this limit is being taken in $\mathcal G_k$. Hence, $C\in\overline{[\X_{tf}]}_k$; furthermore, since $B_X(N)$ in $G$ maps bijectively onto $B_X(N)$ in $C$, $d((G, X), (C, X))\leq \frac1N$. Since $N$ is arbitrary, we get that groups with two conjugacy classes are dense in $\overline{[\X_{tf}]}_k$. Hence a generic group in $\overline{[\X_{tf}]}_k$ has two conjugacy classes.
\end{proof}

Finally, the proof of Corollary \ref{gentopcc} is simply a combination of Corollary \ref{generic} and \cite[Theorem 1.6]{OOK}.


\vspace{1cm}

\noindent {\bf M. Hull:  } MSCS UIC 322 SEO, M/C 249, 851 S. Morgan St.,
Chicago, IL 60607-7045, USA.\\
email: {\it mbhull@uic.edu}

\end{document}

%% file: fig1.pdf_tex
\begingroup%
  \makeatletter%
  \providecommand\color[2][]{%
    \errmessage{(Inkscape) Color is used for the text in Inkscape, but the package 'color.sty' is not loaded}%
    \renewcommand\color[2][]{}%
  }%
  \providecommand\transparent[1]{%
    \errmessage{(Inkscape) Transparency is used (non-zero) for the text in Inkscape, but the package 'transparent.sty' is not loaded}%
    \renewcommand\transparent[1]{}%
  }%
  \providecommand\rotatebox[2]{#2}%
  \ifx\svgwidth\undefined%
    \setlength{\unitlength}{375.92668457bp}%
    \ifx\svgscale\undefined%
      \relax%
    \else%
      \setlength{\unitlength}{\unitlength * \real{\svgscale}}%
    \fi%
  \else%
    \setlength{\unitlength}{\svgwidth}%
  \fi%
  \global\let\svgwidth\undefined%
  \global\let\svgscale\undefined%
  \makeatother%
  \begin{picture}(1,0.28948464)%
    \put(0,0){\includegraphics[width=\unitlength,page=1]{fig1.pdf}}%
    \put(0.34224623,0.04436593){\color[rgb]{0,0,0}\makebox(0,0)[lb]{\smash{$u_1$}}}%
    \put(0.44415784,0.0435686){\color[rgb]{0,0,0}\makebox(0,0)[lb]{\smash{$...$}}}%
    \put(0.55499855,0.04325707){\color[rgb]{0,0,0}\makebox(0,0)[lb]{\smash{$u_k$}}}%
    \put(0,0){\includegraphics[width=\unitlength,page=2]{fig1.pdf}}%
    \put(-0.02953043,0.13391787){\color[rgb]{0,0,0}\makebox(0,0)[lb]{\smash{$r$}}}%
    \put(0.13595106,0.01060592){\color[rgb]{0,0,0}\makebox(0,0)[lb]{\smash{$p_0$}}}%
    \put(0.15199298,0.24735725){\color[rgb]{0,0,0}\makebox(0,0)[lb]{\smash{$q_0$}}}%
    \put(0.34153172,0.1996534){\color[rgb]{0,0,0}\makebox(0,0)[lb]{\smash{$v_1$}}}%
    \put(0.44871162,0.19773149){\color[rgb]{0,0,0}\makebox(0,0)[lb]{\smash{$...$}}}%
    \put(0.55444423,0.20155968){\color[rgb]{0,0,0}\makebox(0,0)[lb]{\smash{$v_k$}}}%
    \put(0.26724167,0.12338653){\color[rgb]{0,0,0}\makebox(0,0)[lb]{\smash{$e_0$}}}%
    \put(0.34524968,0.12427273){\color[rgb]{0,0,0}\makebox(0,0)[lb]{\smash{$e_1$}}}%
    \put(0.55253762,0.12625374){\color[rgb]{0,0,0}\makebox(0,0)[lb]{\smash{$e_k$}}}%
    \put(0.75113511,0.24643283){\color[rgb]{0,0,0}\makebox(0,0)[lb]{\smash{$q_1$}}}%
    \put(0.74592265,0.01630791){\color[rgb]{0,0,0}\makebox(0,0)[lb]{\smash{$p_1$}}}%
    \put(0,0){\includegraphics[width=\unitlength,page=3]{fig1.pdf}}%
  \end{picture}%
\endgroup%